\documentclass[psamsfonts,onesided,12pt]{amsart}
\usepackage{amsmath,amssymb, amsthm, graphicx,hyperref}
\usepackage[all,arc,cmtip]{xy}
\usepackage{enumerate}

\usepackage{tikz}
\usetikzlibrary{arrows,chains,matrix,positioning,scopes}
\makeatletter
\tikzset{join/.code=\tikzset{after node path={%
\ifx\tikzchainprevious\pgfutil@empty\else(\tikzchainprevious)%
edge[every join]#1(\tikzchaincurrent)\fi}}}

\makeatother
\tikzset{>=stealth',every on chain/.append style={join}, every join/.style={->}}

\usepackage{mathrsfs}
\usepackage[margin=1in]{geometry}

\usepackage[inline]{enumitem}   
\makeatletter
\newcommand{\inlineitem}[1][]{%
\ifnum\enit@type=\tw@
    {\descriptionlabel{#1}}
  \hspace{\labelsep}%
\else
  \ifnum\enit@type=\z@
       \refstepcounter{\@listctr}\fi
    \quad\@itemlabel\hspace{\labelsep}%
\fi}
\makeatother

\numberwithin{equation}{section}

\theoremstyle{plain}
\newtheorem{theorem}[equation]{Theorem}
\newtheorem{lemma}[equation]{Lemma}
\newtheorem{proposition}[equation]{Proposition}
\newtheorem{corollary}[equation]{Corollary}

\theoremstyle{definition}
\newtheorem{definition}[equation]{Definition}

\newtheorem{example}[equation]{Example}
\newtheorem*{convention}{Convention}

\theoremstyle{remark}
\newtheorem*{remark}{Remark}

\usepackage[mathscr]{euscript}
\newcommand{\ms}[1]{\mathscr{#1}}

\newcommand{\mc}[1]{\mathcal{#1}}
\newcommand{\mf}[1]{\mathfrak{#1}}
\newcommand{\ov}[1]{\overline{#1}}
\newcommand{\op}[1]{\operatorname{#1}}

\newcommand{\st}{\text{ } \big| \text{ }}

\newcommand{\inv}{^{-1}}

\newcommand{\Z}{\mathbb{Z}} \newcommand{\N}{\mathbb{N}}
 
\newcommand{\A}{\mathbb{A}}
\renewcommand{\H}{\operatorname{H}}

\newcommand{\Hom}{\operatorname{Hom}}
\newcommand{\End}{\operatorname{End}}

\newcommand{\Soc}{\operatorname{Soc}}
\newcommand{\Hd}{\operatorname{Hd}}

\newcommand{\Ker}{\operatorname{Ker}}
\newcommand{\Rad}{\operatorname{Rad}}
\newcommand{\Spec}{\operatorname{Spec}}

\newcommand{\Image}{\operatorname{Im}}

\newcommand{\0}{\overline{0}}
\renewcommand{\1}{\overline{1}}

\newcommand{\define}[1]{\emph{#1}}
\newcommand{\U}[1]{U(\mf{#1})}

\newcommand{\ev}[1]{#1 \mathrm{_{\overline{0}}}}
\newcommand{\od}[1]{#1 \mathrm{_{\overline{1}}}}
\newcommand{\supalg}[1]{\ev{#1} \oplus \od{#1}}

\newcommand{\sL}{\mf{sl}(1|1)}
\newcommand{\gen}[1]{ \langle #1 \rangle}

\newcommand{\rel}[1]{( \mf{#1}, \ev{\mf{#1}})}

\newcommand{\ind}[1]{\U{#1} \otimes_{U(\ev{\mf{#1}})}}

\newcommand{\F}[1]{\mc{F}_{\rel{#1}}}

\newcommand{\sdim}{\op{sdim}}
\newcommand{\rank}{\op{rank}}
\newcommand{\cjt}[1]{M|_{\gen{#1}} \cong k^{\oplus a_1} \oplus P^{\oplus a_2}}
\renewcommand{\P}{\mathbb{P}}
\newcommand{\setzero}{ \{ 0 \} }
\newcommand{\minuszero}{\setminus \setzero}

\begin{document}
\title{%
Lie Superalgebra Modules of Constant Jordan Type }
\author{Andrew J. Talian }
\address{Institut Mittag-Leffler\\ Aurav\"agen 17\\
SE-182 60 Djursholm \\ Sweden}
\curraddr{Department of Mathematics \\ Concordia College \\ Moorhead \\ MN 56562, USA}
\email{atalian@cord.edu} 
\date{June 2013}

\begin{abstract}
The theories of $\pi$-points and modules of constant Jordan
type have been a topic of much recent interest in the
field of finite group scheme representation theory.  These
theories allow for a finite group scheme module $M$ to be
restricted down and considered as a module over a
space of small subgroups whose representation theory is
completely understood, but still provide powerful global
information about the original representation of $M$.

This paper provides an extension of these ideas and
techniques to study finite dimensional supermodules over
a classical Lie superalgebra $\mathfrak{g} = \mathfrak{g}_{\overline{0}} \oplus
\mathfrak{g}_{\overline{1}}$.  Definitions and examples of
$\mathfrak{g}$-modules of
constant super Jordan type are given along with
proofs of some properties of these modules.  Additionally,
endotrivial modules (a specific case of modules of constant Jordan type)
are studied.  The case when $\mathfrak{g}$ is a detecting
subalgebra, denoted $\mathfrak{f}_r$, of a stable Lie superalgebra is considered in
detail and used to construct super vector bundles
over projective space $\mathbb{P}^{r-1}$.  Finally, a complete classification
of supermodules of constant super Jordan type are given
for $\mathfrak{f}_1 = \mathfrak{sl}(1|1)$.
\end{abstract}

\date{}
\maketitle

\section{Introduction}
The study of modules of constant Jordan type was initiated in
\cite{CFP} by Carlson, Friedlander, and Pevtsova
for representations of finite group schemes over
a field of characteristic $p > 0$.
Leading up to the study of these modules
Friedlander and Pevtsova
had previously developed a theory of $p$-points and
the more generalized $\pi$-points in
\cite{FP-ppoints} and \cite{FP-pi} where they study certain maps
$K[t]/t^p \rightarrow KG$ into the group algebra $KG$ of
a finite groups scheme $G$.
The representations of $K[t]/t^p$
can be completely classified by the so called Jordan type
of the matrix representing a $K[t]/t^p$-module.

The simplicity
of the $K[t]/t^p$-modules is used quite successfully to study
$G$ by restricting $KG$-modules along these maps and, under
an equivalence relation, these maps are the so called $\pi$-points.
The set of $\pi$-points, $\Pi(G)$, admits a scheme structure which
is proven to be isomorphic to $\op{Proj} \H^\bullet(G, k)$ and
the $\pi$-points detect projectivity in the sense that a $KG$-module
$M$ is projective if and only if it is projective when restricted
to each $\pi$-point.

An interesting result of $\pi$-point theory was
the introduction of modules of constant Jordan type.  These modules
are ones where the Jordan type (i.e. the isomorphism class of a nilpotent
matrix representing $t$ in the map $K[t]/t^p \rightarrow KG$)
is constant over all $\pi$-points in $\Pi(G)$.
Some of the supporting theory is developed in \cite{FPS} and
in
\cite{CFP} it is shown that these modules are closed under the operations of
direct sums, direct summands, tensor products, dualizing,
and the syzygy operation of Heller shifts.  This class of
module also happens to contain the class of endotrivial modules
which arise as modules of a particular Jordan type, and in a
similar manner to projectivity detection, a $kG$-module is
endotrivial if and only if it is endotrivial when restricted to
each $\pi$-point.

These modules have further gone on to be studied and
applied in \cite{Benson-CJT},
\cite{BP}, \cite{CFS}, and \cite{FP-1} as well as numerous others.
Interesting problems and applications related to these modules are constructing
non-trivial examples, determining that certain Jordan types
cannot exist, and constructing (low rank)
vector bundles through
functors from finitely generated modules to quasi-coherent sheaves.

The main goal of this paper is to extend the theory and results
on modules of constant Jordan type to the field of
Lie superalgebra representation theory.  Let $\mf{g} = \supalg{\mf{g}}$
be a classical Lie superalgebra over an algebraically closed
field $k$ of characteristic 0.
The category of finite dimensional supermodules which are completely reducible over
$\ev{\mf{g}}$, denoted $\F{g}$,
has been widely studied in \cite{BKN2-2009}, \cite{BKN1-2006},
\cite{BKN3-2009},
\cite{B-2002}, \cite{DS-2005}, \cite{LNZ-2011}, \cite{S-2006}, and is the category
of interest in the author's previous work on endotrivial supermodules
in \cite{Talian-2013} and
\cite{Talian-2015}.

There is no general analogue for $\pi$-point theory for an arbitrary
Lie superalgebra, so an alternative method must be found or developed first.
The approach taken in this paper is to introduce the notion
of a supermodule of constant Jordan type using the self commuting cone
$$
\mc{X} = \{ x \in \od{\mf{g}} \st [x,x]=0 \}
$$
of a Lie superalgebra $\mf{g}$
considered by Duflo and Serganova in \cite{DS-2005} which allows
for a natural definition as follows.
For all $x \in \mc{X} \minuszero$, $U(\gen{x}) \cong k[t]/t^2$ and so
restricting a supermodule $M \in \mc{F}$ to $\gen{x}$ allows
us to associate to each point $x \in \mc{X}$ a Jordan type
of $M|_{\gen{x}}$, thus giving a natural definition of supermodules
of constant Jordan type.  This situation is similar
to the study of $KG$-modules of constant Jordan type
over a field $K$ of characteristic 2.  One notable difference
in this setting is that although the Jordan decomposition
at a point $x \in \mc{X}$ has blocks of size either
1 or 2, corresponding to trivial and projective
$U(\gen{x})$-supermodules respectively, the trivial modules
can be concentrated in either even or odd degree
leading to the notion of a super Jordan type.

This definition allows for many results to be proved
about $\mf{g}$-supermodules of constant Jordan type, some
analogous to the finite group schemes setting and some
specific to Lie superalgebras.  Section \ref{S: intro}
gives background information and presents some
of the theory which is used in the remainder of the paper.
Section \ref{S: cjt} defines, introduces first properties,
and gives examples
of modules of constant Jordan type in $\F{g}$.  These modules are shown
to be closed under direct sums, duals, tensor products,
homomorphisms of supermodules, and the syzygy operation,
but strikingly not closed under taking direct summands
(an issue that is further considered and rectified
in Section \ref{S: type f cjt}).  Additionally there are results
strictly relating to Lie superalgebras given.
Proposition \ref{P: CJT has max atyp} shows that
modules of constant Jordan type must lie in a block
of maximal atypicality when $\mf{g}$ is a simple basic
classical Lie superalgebra and Theorem \ref{T: finiteness thm}
shows that, in many cases, the condition for a module to
be of constant Jordan type can be
checked using a finite number of points.

The topic of endotrivial modules is considered in
Section \ref{S: endotrivial}, as again in this setting
these are a specific case of modules of constant Jordan
type.  Theorem \ref{T: endo are CJT 1} shows that
a module $M$ is endotrivial if and only if it is endotrivial
when restricted to each point $x \in \mc{X}$, i.e.
a module $M$ being endotrivial is equivalent to having
constant Jordan type with exactly one trivial summand
(and any number of projective summands).

Section \ref{S: type f cjt} deals with modules of
constant Jordan type over
$$
\mf{f} = \mf{f}_r := \sL \times \dots \sL
$$
where there are $r$ copies of $\sL$.  The Lie superalgebra
$\mf{f}$ is the detecting subalgebra of a stable Lie
superalgebra $\mf{g}$ as introduced in \cite{BKN1-2006}.  These
subalgebras are analogues of elementary abelian subgroups
in finite group representation theory in that they
detect the cohomology of $\mf{g}$ in the sense that
the cohomology ring of $\mf{g}$ embeds into a specific
subring of the cohomology of $\mf{f}$.

When working
in $\F{f}$, it is natural to expand the variety
$\mc{X}$ to include all of $\od{\mf{f}}$ and in doing
so allows a proof of closure of modules of constant
Jordan type under direct summands in this case.
Another benefit of this adjustment is that it allows
for a generalization of the construction of vector
bundles over projective space given in \cite{FP-1}.
The vector bundles constructed from modules of constant
super Jordan type in
Theorem \ref{T: vector bundles} inherit a
natural $\Z_2$ grading from the module structure and are called
algebraic super vector bundles over projective space.
These differ from super vector bundles over a super manifold
(whose construction is explicitly shown to fail in 
Section \ref{SS: super man})
in that the base space is ungraded and is subsequently treated
as a purely even object.

The last major result (Theorem \ref{T: cjt class for f1})
gives a classification of all modules of constant
Jordan type for the specific case of $\mf{f}_1 = \sL$.  The
theorem shows that the only $\mf{f}_1$-modules of constant Jordan
type are in fact the endotrivial modules which are
isomoprhic to the $W$ modules of height 2 and their duals
as defined in \cite{CFS}.  The proof is completed by
exploiting the grading of a $\mf{f}_1$-module and using
the theory of generic kernels and images presented in
\cite{CFS}.  This result also recovers a significant
step (\cite[Theorem 5.12]{Talian-2013}) in the
proof of one of the main results in \cite{Talian-2013}.

\section{Notation and Preliminaries} \label{S: intro}
Let $\mf{g} = \supalg{\mf{g}}$ be a finite dimensional classical
Lie superalgebra
over an algebraically closed field $k$ of characteristic 0.
Here, classical means there is a connected reductive
algebraic group $\ev{G}$ wuch that $\op{Lie}(\ev{G}) = \ev{\mf{g}}$,
and an action of $\ev{G}$ on $\od{\mf{g}}$ that differentiates
to the adjoint action of $\ev{\mf{g}}$ on $\od{\mf{g}}$.
Note, if $\mf{g}$ is classical, then
$\ev{\mf{g}}$ is reductive as a Lie
algebra and $\od{\mf{g}}$ is semisimple as a $\ev{\mf{g}}$-module,
but it is not assumed that $\mf{g}$ is simple.
There
are a number of module categories associated to $\mf{g}$ but one
category, denoted $\F{g}$, is
of particular interest because of its desirable
properties and is defined as follows.

Let $\mf{t} \subseteq \mf{g}$ be Lie superalgebras.  Then define
$\mc{F}_{(\mf{g}, \mf{t})}$ to be the full subcategory of
finite dimensional $\mf{g}$-supermodules which are completely
reducible over $\mf{t}$.  This category has enough projectives
and injectives, and is self injective.
Of primary concern is $\mc{F} := \F{g}$, as
it has been widely studied, and a number of the results on
this category are used here.
In the case that $\ev{\mf{g}}$ is
a semisimple Lie algebra, then this is simply the category of all
finite dimensional $\mf{g}$-supermodules.  By convention, supermodules
are the only objects considered here and thus
may be referred to as modules from now on.

\subsection{Atypicality and Blocks}

Much of the theory in what follows does not rely on the
Lie superalgebra $\mf{g}$ having a non-degenerate bilinear
form, i.e. is basic classical,
but when one does exist (in all cases except for
the Type I superalgebra $P(n)$), then we may say a bit
more about modules of constant Jordan type by making
observations about certain combinatorial invariants.

Let $\mf{g}$ be a basic classical Lie superalgebra whose
non-degenerate bilinear form is denoted as
$( \ , \ ): \mf{g} \times \mf{g} \rightarrow k$.
By fixing a
Cartan subalgebra $\mf{h} \subseteq \ev{\mf{g}}$, $\mf{g}$
can be decomposed into root spaces, indexed by $\Phi$
$$
\mf{g} = \mf{h} \oplus \bigoplus_{\alpha \in \Phi} \mf{g}_{\alpha}
$$
each of which are homogeneous and one dimensional.  Thus,
we can define a parity associated to each root by
assigning it to be the parity of the corresponding root space,
and denote this decomposition
by writing $\Phi = \ev{\Phi} \sqcup \od{\Phi}$.
We also fix a choice of a Borel subalgebra $\mf{b}$
containing $\mf{h}$ which defines positive and negative roots.  Then define
$$
\rho =\dfrac{1}{2}\sum_{\ev{\alpha} \in \Phi^+_{\ov{0}}} \ev{\alpha}
- \dfrac{1}{2}\sum_{\od{\alpha} \in \Phi^+_{\ov{1}}} \od{\alpha}
$$
where $\Phi^+_{i}$ denotes the positive roots in $\Phi_{i}$
relative to the choice of $\mf{b}$.


The form $( \ , \ )$ induces a form on $\mf{h}^*$ (denoted
in the same way) and thus on the roots as well.
We say
that a root $\alpha$ (necessarily odd) is \emph{isotropic} if
$(\alpha, \alpha) = 0$.  Then the maximal number of isotropic,
pairwise orthogonal roots is called the \emph{defect of $\mf{g}$}
denoted $\op{def}(\mf{g})$
and does not depend on the choice of $\mf{h}$, hence is
well defined.  Additionally, for for a weight
$\lambda \in \mf{h}^*$, the maximal number of isotropic,
pairwise orthogonal roots $\alpha$ such that
$(\lambda + \rho, \alpha) = 0$ is called the
\emph{atypicality of $\lambda$} and denoted as
$\op{atyp}(\lambda)$.  This number again does not depend on
any choices made and so is also well defined.  By construction,
$\op{atyp}(\lambda) \leq\op{def}(\mf{g})$ for any weight
$\lambda$.

When $\mf{g}$ is also Type I, then $\F{g}$ is a highest
weight category and the simple modules are parameterized by
a set $X^+ \subseteq \mf{h}^*$ and so any simple module
is a highest weight module and is denoted
$L(\lambda)$ where $\lambda \in X^+$.  In this case we
define the \emph{atypicality of $L(\lambda)$} to be the
atypicality of the weight $\lambda$ and write
$\op{atyp}(L(\lambda))$.

If $Z$ denotes the center of the universal enveloping algebra
$\U{g}$, then for any weight $\lambda$ we can define an
algebra homomorphism $\chi_{\lambda} : Z \rightarrow k$ which
defines a central character.  These characters decompose
the category $\mc{F}$ into blocks
$$
\mc{F} = \bigoplus \mc{F}^{\chi}
$$
in the sense that any indecomposable module $M$ is contained
in exactly one $\mc{F}^{\chi}$ and we say that $M$ admits
the character $\chi$.  The atypicality of a central character
$\chi$ is the maximal atypicality of a weight $\lambda$ such
that $\chi_{\lambda} = \chi$.  One important property of
these blocks is that \cite{GS-2010}
gives an equivalence of blocks of finite dimensional
$\mf{g}$-modules which admit certain characters.
\begin{theorem}[Gruson-Serganova]
Let $\lambda$ be a dominant weight with atypicality $\ell$,
then the block $\mc{F}^{\chi_\lambda}$ is equivalent
to the maximal atypical block of $\mf{g}_\ell$ containing
the trivial module, where
\begin{itemize}
\item if $\mf{g} = \mf{gl}(m|n)$ then $\mf{g}_\ell = \mf{gl}(\ell | \ell)$;

\item if $\mf{g} = \mf{osp}(2m + 1| 2n)$ then $\mf{g}_\ell = \mf{osp}(2\ell + 1 | 2\ell);$

\item if $\mf{g} = \mf{osp}(2m|2n)$ then $\mf{g}_\ell = \mf{osp}(2\ell | 2\ell)$ or $(2\ell + 2 | 2\ell).$
\end{itemize}
\end{theorem}
Thus up to some equivalence, modules of a fixed atypicality
(see Proposition \ref{P: CJT has max atyp})
lie in the principal block of a fixed Lie superalgebra.

\subsection{The Associated Variety}

The representation theory of the one dimensional abelian
Lie superalgebra generated by one odd element is completely
understood.  As verified by direct calculation in
\cite[Section 5.2]{BKN1-2006}, if $\mf{g} = \od{\mf{g}} = \gen{x}$ is
a one dimensional Lie superalgebra, then
there are only four indecomposable non-isomorphic
modules (or two if the parity change is ignored).  If
$\Pi$ denotes the parity change functor and $k_{ev}$ the trivial
module concentrated degree $\ov{0}$ and $k_{od} = \Pi(k_{ev})$, then
the four indecomposable $U(\gen{x})$ modules are $k_{ev}$ and
$k_{od}$, and their two dimensional
projective covers which are isomorphic to $U(\gen{x})$ and
$\Pi(U(\gen{x}))$ respectively.

Dufflo and Sergonova define the
variety
$$
\mc{X} = \{ x \in \od{\mf{g}} \st [x,x]=0 \}
$$
in \cite{DS-2005}.  If we view $\od{\mf{g}}$ as an affine
variety endowed with the Zariski topology, then
$\od{\mf{g}} \cong \A^{\dim{\od{\mf{g}}}}$, and
$\mc{X}$ is a Zariski closed $\ev{G}$-invariant cone in
$\od{\mf{g}}$ and is referred to as the \emph{self commuting cone}.
The condition that $[x,x]=0$ is equivalent
to the condition that $x^2 = 0$ in the universal enveloping
algebra $\U{g}$ and so restricting a module $M \in \F{g}$ to
an element $x \in \mc{X}$ is reminiscent of the notion of
Friendlander and Pevtsova's $\pi$-points (see \cite{FP-pi}).

For $M \in \F{g}$, the authors of \cite{DS-2005} further define a
subvariety of $\mc{X}$ by viewing multiplication by
$x \in \mc{X}$ as a linear map from $M$ to itself and define
$M_x = \Ker(x)/ \Image(x)$ which is called the fiber of $M$
at $x$ and set
$$
\mc{X}_M = \{ x \in \mc{X} \st
 M_x \neq 0 \}.
$$
By the characterization of the representations of $\gen{x}$
for $x \in \mc{X}$ given above, it can be seen that the
previous definition is equivalent to the following
\begin{equation} \label{Eq: alt def of assoc var}
\mc{X}_M = \{ x \in \mc{X} \st M|_{\gen{x}} \text{ is not
projective as a $U(\gen{x})$-module} \} \cup \{0\}.
\end{equation}
The so called \emph{associated variety} $\mc{X}_M$ is shown
to detect projectivity when $\mf{g}$ is a simple classical
Lie superalgebra
in \cite[Theorem 3.4]{DS-2005}
in the sense that $M$ is
a projective module in $\F{g}$ if and only if $\mc{X}_M = \{ 0 \}$.
When $M$ is finite dimensional, the case of interest in this
paper, $\mc{X}_M$ is also a Zariski closed $\ev{G}$-invariant
subcone of $\mc{X}$.

Furthermore, the associated varieties satisfy certain
desirable properties
as proven in \cite{DS-2005}.  Recall that for a super vector
space
$V = \supalg{V}$, we define the superdimension of $V$
to be $\sdim(V) := \dim \ev{V} - \dim \od{V}$.

\begin{proposition}[Duflo-Serganova] \label{P: DS}
Let $M,N \in \mc{F}$ and $x \in \mc{X}$.  Then
\begin{enumerate}
\item if $M = \ind{g} M_0$ for some $\ev{\mf{g}}$-module $M_0$,
then $\mc{X}_M = \{ 0 \}$;

\item for the trivial module $k$ concentrated in even or
odd degree, $\mc{X}_k = \{ 0 \}$;

\item $\mc{X}_{M \oplus N} = \mc{X}_M \cup \mc{X}_N$;

\item $\mc{X}_{M \otimes N} = \mc{X}_M \cap \mc{X}_N$;
\label{P: tensoring is interscting}

\item $\mc{X}_M = \mc{X}_{M^*}$;

\item $\sdim(M) = \sdim(M_x)$. \label{P: sdim M = sdim Mx}
\end{enumerate}
\end{proposition}

The variety $\mc{X}$ may also be related to the combinatorial
invariants in the previous section by the following.
Let $S$ denote all sets of isotropic,
pairwise orthogonal roots $\alpha$, and let
$$
S_m = \{A \in S \st |A| = m\}.
$$
Then $S_m$ is nonempty
for $1, 2, \dots, \op{def}(\mf{g})$.  It is shown in
\cite{DS-2005} that for any $x \in \mc{X}$ the $\ev{G}$ orbit
of $x$ contains an element of the form $x_1 + \dots + x_m$
where $x_i \in \mf{g}_{\alpha_i}$ where
$\{\alpha_1, \dots, \alpha_m \} \in S$ and that this number
$m$ only depends on the orbit.  Then define $m$ to be
the \emph{rank of $x$} which is denoted $\op{rank}(x)$.

Note that the fiber $M_x$ has a module structure over a much larger
superalgebra defined in \cite[Section 6]{DS-2005}, denoted
$\mf{g}_x$.  However, $\mf{g}_x$ is too large for
consideration here.  For example, if
$\mf{g} = \mf{gl}(m|n)$, and $x$ has rank $k$ then
$\mf{g}_x \cong \mf{gl}(m-k|n-k)$.  Thus, we always
consider $M_x$ simply as an $\gen{x}$-module.

%

\subsection{Constant Jordan Type in Finite Group Schemes}
In \cite{CFP}, for a field $k$ of characteristic $p > 0$,
Carlson, Friedlander, and Pevtsova
use the simplicity of the
representation theory of $k[t]/t^p$-modules and a number of
cohomological results define and study what is called
the Jordan type of a $k[t]/t^p$-module and eventually
the Jordan type of a module for a finite group scheme if
the module satisfies certain conditions.

If we fix
a basis for $M$ a $k[t]/t^p$-module of dimension $n$,
then the associated
representation $\rho : k[t]/t^p \rightarrow M_n(k)$
is completely determined by $\rho(t)$ which is some
$p$-nilpotent $n \times n$ matrix.  Furthermore, by a change
of basis we can write $\rho(t)$ in Jordan canonical form.
So up to isomorphism, the structure of $M$ is completely
determined by the sizes of the Jordan blocks of $\rho(t)$
which are necessarily of size $\leq p$, i.e. a $p$-restricted
partition of $n$.  Thus we may associate to $\rho$ a \emph{Jordan
type}, $a_1[1] + a_2[2] + \dots + a_{p-1}[p-1] + a_p[p]$ where
$a_i$ denotes the number of blocks of size $i$ in the
partition of $n$ associated to $M$.

For a finite group scheme $G$,
the authors consider maps of
$K$-algebras $\alpha : K[t]/t^p \rightarrow KG$ which
satisfies certain conditions, where $K/k$ is a field
extension.  There is an equivalence condition placed on these
maps and each equivalence class is referred to as
a $\pi$-point.  The space of $\pi$-points admits
a scheme structure which is isomorphic to
$\op{Proj} \H^\bullet(G,k)$.  If $M$ is a finite
dimensional $kG$-module, then restricting along $\alpha$
allows us to define the Jordan type of $\alpha^*(M)$ at
a $\pi$-point $\alpha$, and
if the Jordan type is the same over all $\pi$-points, then
$M$ is said to be of \emph{constant Jordan type}.

\section{Constant Jordan Type} \label{S: cjt}
Modules of constant Jordan type for finite group schemes
are studied extensively
in \cite{CFP} and applied in \cite{Benson-CJT},
\cite{BP}, and  \cite{CFS} as well as in numerous others.  With
this motivation and the observation that
$U(\gen{x}) \cong
k[t]/t^2$ for $x \in \mc{X}$, we begin defining an analogous theory for
Lie superalgebras.

\subsection{Definitions and Conventions}

Let $\mf{g} = \supalg{\mf{g}}$ be a finite dimensional Lie
superalgebra such that $\mc{X}$ spans $\od{\mf{g}}$.  
While the following definitions make sense
for arbitrary Lie superalgebras (when $\mc{X} \neq \{0\}$),
many of the properties and
results rely on using the imposed restrictions on $\mf{g}$,
particularly the use of \cite[Theorem 3.4]{DS-2005}.

The condition imposed is not very restrictive
considering $\mc{X}$ spans $\od{\mf{g}}$ for
all simple classical Lie superalgebras except
$\mf{osp}(1|2n)$ in which case $\mc{X} = \{0\}$ and all finite
dimensional modules are projective trivializing the
theory.  It is clear then that the spanning property is preserved
under products of Lie superalgebras as well giving further
instances of application.

By taking advantage
of the explicit description of all indecomposable $U(\gen{x})$-modules we make the following definitions.

\begin{definition} \label{D: JT at x}
Let $M \in \F{g}$ and let $\mc{X}$ be the cone of
self commuting elements of $\mf{g}$.
For $x \in X \setminus \{0\}$,
$$
M|_{\gen{x}} \cong
k_{ev}^{\oplus a_{ev}} \oplus k_{od}^{\oplus a_{od}} \oplus P^{\oplus a_2},
$$
where $P$ is, up to parity change, the unique projective
indecomposable $U(\gen{x})$-module, and
by definition $a_{ev} + a_{od} + 2a_2 = \dim(M)$.
Then we then define the \define{super Jordan type of $M$ at $x$} to be
the isomorphism type of $M|_{\gen{x}}$ and denote it by
$(a_{ev}| a_{od})[1] + a_2[2]$.

When there is
no need to distinguish the particular dimensions $a_{ev}$
and $a_{od}$, then we set $a_1 = a_{ev} + a_{od}$ and
define the \define{Jordan type of $M$ at $x$} to
be the isomorphism type of
$$
M|_{\gen{x}} \cong k^{\oplus a_1}
\oplus P^{\oplus a_2}
$$
and denote it by $a_{1}[1] + a_2[2]$.
\end{definition}

\begin{definition} \label{D: CJT}
Let $M \in \F{g}$ and let $\mc{X}$ be the the cone of
self commuting elements of $\mf{g}$.
We say that $M$ is of \define{constant (super) Jordan type} if
the (super) Jordan type of
$M$ at $x$ is the same for all $x \in \mc{X}\setminus \{0\}$.

When $M$ is of constant (super) Jordan type, we define the
\define{(super) Jordan type of $M$} to be the (super) Jordan type
of $M$ at $x$ for any point $x \in \mc{X} \setminus \{0\}$.
\end{definition}

\begin{remark} \label{R: reduction to a1}
Because the super Jordan type of a module $M$, either at a point or for all
of $M$, is indexed by the numbers $a_{ev}$, $a_{od}$, and $a_2$ and
we can express them as an equation $a_{ev} + a_{od} + a_2 = \dim(M) = n$,
where $n$ is given by an intrinsic property of $M$, we may refer simply
refer to the super Jordan type as
$(a_{ev}|a_{od})[1]$ and the Jordan type as
$a_1[1]$ or just the single non-negative
integer $a_1$.  In the language of
\cite{CFP}, if $M$ and $N$ have super Jordan types
$(a_{ev}|a_{od})[1] + a_2[2]$ and
$(b_{ev}|b_{od})[1] + b_2[2]$ these are called
\define{stably equivalent} if
$a_{ev} = b_{ev}$ and $a_{od} = b_{od}$ because in the stable category of
$\gen{x}$-supermodules, $[M] = [N]$ and $(a_{ev}|a_{od})$ ($a_1$ respectively)
is called the
\emph{stable super Jordan type} (\emph{stable Jordan type} respectively) of $M$.

\end{remark}

An interesting consequence of the additional information of
the super dimension of a module is that there is no distinction
between modules of constant super Jordan type and modules
of constant Jordan type, as seen in the corollary
following the lemma.

\begin{lemma} \label{L: fiber is Omega 0}
Let $M \in \mc{F}$ and $x \in \mc{X} \setminus \{0\}$.  Then
$M_x \cong \Omega^0(M|_{\gen{x}})$ as $\gen{x}$-supermodules.
\end{lemma}
\begin{proof}
Because of the explicit
description of $\gen{x}$-supermodules, we can see
that when viewing multiplication by $x$ as a linear operator
on $M$,
\begin{gather*}
\Ker(x) = \Soc(M|_{\gen{x}}) = k_{ev}^{\oplus a_{ev}} \oplus k_{od}^{\oplus a_{od}} \oplus \Soc(P)^{\oplus a_2} \\
\Image(x) = \Soc(P)^{\oplus a_2}
\end{gather*}
where $M|_{\gen{x}} \cong k_{ev}^{\oplus a_{ev}} \oplus k_{od}^{\oplus a_{od}} \oplus P^{\oplus a_2}$ is the
Jordan type of $M$ at $x$.  Then it is clear that
$M_x \cong k_{ev}^{\oplus a_{ev}} \oplus k_{od}^{\oplus a_{od}}$ as an $\gen{x}$-supermodule, and the claim is proven.
\end{proof}

\begin{corollary} \label{C: fiber is Omega 0}
Let $M \in \mc{F}$ and $x \in \mc{X} \setminus \{0\}$.  Then
\begin{enumerate}
\item in the stable module category of $\gen{x}$-supermodules,
$[M_x] = [M|_{\gen{x}}]$;
\label{p1}

\item the stable Jordan type of $M$ at $x$ is $a_1 = \dim(M_x)$ and
$M$ is of constant Jordan type if and only if this
equation holds for any $x \in \mc{X} \setminus \{0\}$; \label{p2}

\item if $M$ has stable Jordan type $a_1$ at a point $x$, then
the stable super Jordan type is uniquely determined by
$\sdim(M)$; \label{p3}

\item $M$ is of constant Jordan type if and only if $M$ is of constant
super Jordan type. \label{p4}
\end{enumerate}
\end{corollary}
\begin{proof}
Parts (\ref{p1}) and (\ref{p2}) follow immediately from Lemma
\ref{L: fiber is Omega 0}.  For (\ref{p3}), recall
Proposition \ref{P: DS} (\ref{P: sdim M = sdim Mx}).  Since
$\sdim(M) = \sdim(M_x)$ is
independent of the choice of $x$ and stable the Jordan type of $M$ is
$a_1 = \dim(M_x)$, then we have
$a_{ev} + a_{od} = \dim(M_x)$ and $a_{ev} - a_{od} =\sdim(M_x)$,
and thus
$$
a_{ev} = \dfrac{\dim(M_x) + \sdim(M_x)}{2} \quad \text{and} \quad
a_{od} = \dfrac{\dim(M_x) - \sdim(M_x)}{2}
$$
for any $x \in \mc{X} \setminus \{0\}$.

Then part (\ref{p4}) follows directly from part (\ref{p3}).
\end{proof}

\begin{convention}
Given Corollary \ref{C: fiber is Omega 0} (\ref{p4}), we will only use the
term modules of constant Jordan type since the ``super'' condition follows
automatically.
\end{convention}

\subsection{Properties of Modules of Constant Jordan Type}
\label{S: Properties of CJT}
We now seek to identify and understand how these
modules behave under standard categorical operations.
With the goal of identification, we observe the following
properties of these modules.

\begin{proposition} \label{P: CJT has max atyp}
Let $M \in \F{g}$ be a module of constant Jordan type.  Then
either $\mc{X}_M = \{ 0\}$ or $\mc{X}_M = \mc{X}$.
Additionally, $\mc{X}_M = \{0\}$ is equivalent to the
statement that $M$ is projective or that $M$ has stable
Jordan type $a_1 = 0$.

Furthermore, if $\mf{g}$ is simple basic classical and $M$
is indecomposable and is not projective, then $M$ lies in a
block of maximal atypicality, i.e., $M$ has atypicallity
$d = \op{def}(\mf{g})$.
\end{proposition}
\begin{proof}
If $M$ has stable Jordan type $a_1$ then by Corollary
\ref{C: fiber is Omega 0} (\ref{p2}), $a_1 = \dim(M_x)$ for
all $x \in \mc{X}$.  Then by the definition of $\mc{X}_M$,
if $a_1 = 0$ then $\mc{X}_M = \{ 0\}$, and $\mc{X}_M = \mc{X}$
otherwise.

By \cite[Theorem 3.4]{DS-2005}, $M$ is projective in
$\mc{F}$ if and only if $\mc{X}_M = \{ 0 \}$ and this is
equivalent to $M$ having stable Jordan type $a_1 = 0$
by definition.

Let $\mc{X}_k = \{ x \in \mc{X} \st \rank(x) = k \}$.  Then
$\bar{\mc{X}_k} = \bigcup_{i \leq k} \mc{X}_i$ and
$\bar{\mc{X}_d} = \mc{X}$ where $d$ is the defect of $\mf{g}$.
Then by \cite[Theorem 5.3]{DS-2005}, $\mc{X}_M \subseteq \bar{\mc{X}_\ell}$ where $M$ has atypicality $\ell$.
Since $M$ is not projective by assumption, $\mc{X}_M \neq \{0\}$, and so it must be that
$\mc{X}_M = \mc{X}$ which implies that $\ell = d$.
\end{proof}


\begin{lemma} \label{L: proj restriction}
Let 
$Q \in \mc{F}$.  Then $Q$ is projective in $\mc{F}$
if and only if $Q|_{\gen{x}}$
is projective for all $x \in \mc{X}$.
\end{lemma}
\begin{proof}
Since $Q$ is projective if and only if $\mc{X}_Q = \{0\}$ and
the alternative definition of $\mc{X}_Q$ in Equation~\ref{Eq: alt def of assoc var} implies that
this is equivalent to the condition that
$Q|_{\gen{x}}$ is projective for all $x \in \mc{X}$.
\end{proof}

In the following proposition, in parts (\ref{q1})--(\ref{q4}), there are straightforward
dimension formulas (in addition to the ones for super dimension given)
and so the super Jordan type of the modules is completely determined.
However, in part (\ref{q5}) there is no general
formula for determining $\dim(\Omega^n(M))$
and so we may only determine the super Jordan type of $\Omega^n(M)$ up to
stable equivalence.

\begin{proposition} \label{P: closure of CJT}
Let $M, N \in \mc{F}$ be modules of constant Jordan type $a_1[1] + a_2[2]$
and $b_1[1] + b_2[2]$ respectively. Then
\begin{enumerate}
\item $M \oplus N$ is of constant Jordan type $(a_1 + b_1)[1] + (a_2 + b_2)[2]$
and
$\sdim(M \oplus N) = \sdim(M) + \sdim(N);$
\label{q1}
\item $M^*$ is of constant Jordan type $a_1[1] + a_2[2]$ and
$\sdim(M^*) = \sdim(M)$;
\label{q2}

\item $M \otimes N$ is of constant Jordan type $(a_1 \cdot b_1)[1] + (a_1 \cdot b_2 + a_2 \cdot b_1 + a_2 \cdot b_2)[2]$
and $\sdim(M \otimes N) = \sdim(M) \cdot \sdim(N)$;
\label{q3}

\item $\Hom_k(M,N)$ is of constant Jordan type $(a_1 \cdot b_1)[1] + (a_1 \cdot b_2 + a_2 \cdot b_1 + a_2 \cdot b_2)[2]$
and $\sdim(\Hom_k(M,N)) =  \sdim(M) \cdot \sdim(N);$
\label{q4}

\item $\Omega^n(M)$ is of stable Jordan type $a_1$ for all
$n \in \Z$ and $\sdim(\Omega^n(M)) = (-1)^n \sdim(M)$.
\label{q5}
\end{enumerate}
\end{proposition}
\begin{proof}
By \cite[Proposition 3.5 (g)]{Talian-2013}, and Lemma
\ref{L: fiber is Omega 0}, since $\dim(M_x) = a_1$ and
$\dim(N_x) = b_1$ for any $x \in \mc{X} \setminus \setzero$, then
$\dim((M \oplus N)_x) = \dim(\Omega^0((M \oplus N)|_{\gen{x}})) =  a_1 + b_1$ which shows (\ref{q1}).

For (\ref{q2}), $M^* = \Hom_k(M, k_{ev})$, so then
\begin{gather*}
M^*|_{\gen{x}} = \Hom_k(M, k_{ev})|_{\gen{x}} \cong
\Hom_k(M|_{\gen{x}}, k_{ev}) \cong 
\Hom_k(k^{\oplus a_1} \oplus P^{\oplus a_2}, k_{ev})  \\
\cong \Hom_k(k^{\oplus a_1}, k_{ev}) \oplus \Hom_k(P^{\oplus a_2}, k_{ev}) \cong 
k^{\oplus a_1} \oplus P^{\oplus a_2}
\end{gather*}
and so if $M$ has super Jordan type
$(a_{ev}| a_{od})[1]$ at $x$ then $M^*$ also has
super Jordan type $(a_{ev}| a_{od})[1]$ and
$\sdim(M_x) = \sdim(M^*_x)$.  
This is again constant over all $x$ and so $M^*$ has
constant Jordan type $a_{ev} + a_{od} = a_1$.

By \cite[Proposition 3.5 (f)]{Talian-2013} and again
using Lemma
\ref{L: fiber is Omega 0}, $(M \otimes N)_x
= \Omega^0((M \otimes N)|_{\gen{x}})$ as $\gen{x}$-modules.
Then $[(M \otimes N)|_{\gen{x}}] = [k^{\oplus a_1} \otimes
k^{\oplus b_1}]$ and so $\dim((M \otimes N)_x) = a_1 \cdot b_1$
for any $x \in \mc{X} \minuszero$ and $M \otimes N$ is of constant
Jordan type $a_1 \cdot b_1$ and (\ref{q3}) is established.

Part (\ref{q4}) follows directly from (\ref{q2}) and (\ref{q3})
by the canonical isomorphism $\Hom_k(M,N) \cong N \otimes M^*$.

Finally, for (\ref{q5}) let
$$
\begin{tikzpicture}[start chain] {
	\node[on chain] {$0$};
	\node[on chain] {$\Omega^1(M) $} ;
	\node[on chain] {$Q$};
	\node[on chain] {$M$};
	\node[on chain] {$0$}; }
\end{tikzpicture}
$$
be the exact sequence which defines $\Omega^1(M)$.  Then
$Q|_{\gen{x}}$ is the projective by Lemma \ref{L: proj restriction} and restriction
is an exact functor, so
$$
\begin{tikzpicture}[start chain] {
	\node[on chain] {$0$};
	\node[on chain] {$\Omega^1(M)|_{\gen{x}} $} ;
	\node[on chain] {$Q|_{\gen{x}}$};
	\node[on chain] {$M|_{\gen{x}}$};
	\node[on chain] {$0$}; }
\end{tikzpicture}
$$
is exact as well.  Thus, $\Omega^1(M)|_{\gen{x}} \cong
\Omega^1(M|_{\gen{x}}) \oplus R$ where $R$ is a projective
$\gen{x}$-module.  Since $M|_{\gen{x}} \cong k^{\oplus a_1}
\oplus P^{\oplus a_2}$ and and syzygies commute with direct sums
and are trivial on projective summands, we have that
$$
\Omega^1(M)|_{\gen{x}} \cong \Omega^1(k^{\oplus a_1}) \oplus R
\cong \Pi(k^{\oplus a_1}) \oplus R.
$$
Thus $\Omega^1(M)$ has constant super Jordan type
$(a_{od}| a_{ev})[1]$ and Jordan type $a_1$.
A dual argument shows the same for $\Omega^{-1}(M)$ and induction
shows that $\Omega^n(M)$ is of constant Jordan type $a_1$
for any $n \in \Z$ and $\sdim(\Omega^n(M)) = (-1)^n \sdim(M)$.
\end{proof}

The previous proposition gives many ways of constructing new
modules of constant Jordan type from already known modules
of constant Jordan type.  In fact for any choice of $\mf{g}$, because the
Jordan types in this setting are so simple,
it trivializes the question
of which Jordan types are realized up stable equivalence
because we may take
$M = k^n$ for $n \in \N$ and chose the concentration
of the trivial
summands to realize any stable Jordan type.  Thus the question of
realization must be
modified to include indecomposability or to consider the
full Jordan type.

The definition of Jordan type depends \emph{a priori}
on checking the Jordan type of $M$ at all points $x \in \mc{X}$
of which there are infinitely many.  However, when $\mf{g}$
is a basic classical Lie superalgebra with indecomposable
Cartan matrix,
\cite[Theorem 4.2]{DS-2005} shows that there are finitely
many $\ev{G}$-orbits in $\mc{X}$ and the Jordan
type is invariant along the orbit since the modules
in the same orbit are isomorphic via twisting by the adjoint
action.  This immediately proves the following theorem.  See
\cite[Remark 4.1]{DS-2005} for more on which Lie superalgebras
satisfy the condition (which includes the simple basic classical
types).

\begin{theorem} \label{T: finiteness thm}
Let $\mf{g}$ be a basic classical Lie superalgebra with indecomposable
Cartan matrix.  Then $M \in \F{g}$ is of constant Jordan type if and
only if the Jordan type of $M$ is constant over a finite number of
a finite set of
points which are orbit representatives of the $\ev{G}$ action on $\mc{X}$.
\end{theorem}

\subsection{Examples} \label{SS: examples}
We consider some examples and non-examples here which further
illustrate the structure and properties of modules
of constant Jordan type.

\begin{example} 
Since $k_{ev}$ and $k_{od}$ are of constant Jordan type
$1[1]$, in cases when these modules have nonzero complexity
over $\mf{g}$,
by Proposition \ref{P: closure of CJT} (\ref{q5}),  $\Omega^n(k)$ give infinitely
many nonisomorphic 
modules of stable Jordan type $1[1]$ as well.
The reader who is familiar
with endotrivial modules will note that each of
these syzygies is endotrivial as well.  This topic
is discussed further in Section \ref{S: endotrivial}, and
Section \ref{SS: syzygy mod soc}
gives a construction of an example of modules of constant
Jordan type which are not endotrivial.
\end{example}

The next examples show that, unfortunately, modules
of constant Jordan type are
not closed under taking direct summands.

\begin{example} \label{Ex: 1}
Let $\mf{g} = \mf{sl}(1|1)$ where
$\ev{\sL} = \{t\}$ and $\od{\sL} =\{x, y\}$ are bases
for their respective components and the only nontrivial bracket
is $[x,y] = xy + yx = t$.  Then one can verify directly that
$\mc{X} = k\cdot x \cup k \cdot y$ in this case.  Then under
the action of $\ev{G} \cong k$ (see
the remark at the end of the previous section),
there are two orbits whose
representatives are chosen to be $x$ and $y$ for simplicity.

Let $K(0)$ and $K^-(0)$
be the Kac and dual Kac module of the trivial module $k_{ev}$,
as defined in \cite[Section 3.1]{BKN3-2009}.  Then
$K(0)$ and $K^-(0)$ are both two dimensional modules with
a one dimensional submodule as the socle and a one dimensional
head.  The structure of the modules is given by
\begin{equation*}
K(0):
\begin{tikzpicture}[description/.style={fill=white,inner sep=2pt},baseline=(current  bounding  box.center)]

\matrix (m) [matrix of math nodes, row sep=3em,
column sep=2.5em, text height=1.5ex, text depth=0.25ex]
{ k_{\text{ev}}\\
  k_{\text{od}} \\ };

\path[-,font=\scriptsize]
	(m-1-1) edge node[auto] {$ y $} (m-2-1) ;
\end{tikzpicture}
\qquad \qquad \qquad
K^-(0):
\begin{tikzpicture}[description/.style={fill=white,inner sep=2pt},baseline=(current  bounding  box.center)]

\matrix (m) [matrix of math nodes, row sep=3em,
column sep=2.5em, text height=1.5ex, text depth=0.25ex]
{ k_{\text{od}}\\
  k_{\text{ev}} \\ };

\path[-,font=\scriptsize]
	(m-1-1) edge node[auto] {$ x $} (m-2-1) ;
\end{tikzpicture}
\end{equation*}
where $t$ acts trivially on all of $K(0)$ and $K^-(0)$.
Then $K(0)$ is projective as a $\gen{y}$-module and
$K^-(0)$ is projective as a $\gen{x}$-module and
$\mc{X}_{K(0)} = k \cdot x$ and $\mc{X}_{K^-(0)} = k \cdot y$,
and so neither module is of constant Jordan type by
Proposition \ref{P: CJT has max atyp}.
However, we
see that $\mc{X}_{K(0) \oplus K^-(0)} = k\cdot x \cup k \cdot y$
and in fact $K(0) \oplus K^-(0)$ is a module of constant
Jordan type $2[1] + 1[2]$.
\end{example}

\begin{example} \label{Ex: 2}
Furthermore, considering the setting of the last example,
define
$$
M = k_{ev} \oplus K(0) \quad \text{and} \quad N = k_{ev} \oplus K^-(0).
$$
Then we have that both $\mc{X}_M = \mc{X}$ and $\mc{X}_N =
\mc{X}$ and $M \oplus N$ is of constant Jordan type $4[1] + 1[2]$,
but neither $M$ nor $N$ is of constant Jordan type.
This shows that even in a situation with the possibly stronger
assumption that the module summands have maximal associated
varieties, modules of constant Jordan type are not closed
under direct summands.
\end{example}

\begin{example} \label{Ex: 3}
Again, let $\mf{g} = \mf{sl}(1|1)$ and define $M$ and $N$ to
be four dimensional modules with structures given by
\begin{equation*}
M :
\begin{tikzpicture}[description/.style={fill=white,inner sep=2pt},baseline=(current  bounding  box.center)]

\matrix (m) [matrix of math nodes, row sep=1.3em,
column sep=1em, text height=1.5ex, text depth=0.25ex]
{ & k_{ev} & & k_{ev} \\
  k_{od} & & k_{od} & \\ };

\path[-,font=\scriptsize]
	(m-1-2) edge node[auto, swap] {$ x $} (m-2-1)
			edge node[auto] {$ y $} (m-2-3)
	(m-1-4) edge node[auto, swap] {$ x $} (m-2-3);
\end{tikzpicture}
\qquad \qquad
N :
\begin{tikzpicture}[description/.style={fill=white,inner sep=2pt},baseline=(current  bounding  box.center)]

\matrix (m) [matrix of math nodes, row sep=1.3em,
column sep=1em, text height=1.5ex, text depth=0.25ex]
{ k_{ev} & & k_{ev} & \\
   & k_{od} & & k_{od} \\ };

\path[-,font=\scriptsize]
	(m-1-1) edge node[auto] {$ y $} (m-2-2)		
	(m-1-3) edge node[auto, swap] {$ x $} (m-2-2)
			edge node[auto] {$ y $} (m-2-4);
\end{tikzpicture}
\end{equation*}
and all other actions are zero.  Then $M \oplus N$ is of
constant Jordan type $2[1] + 3[2]$, and $M$ and $N$ are indecomosable,
but neither $M$ nor $N$ is of constant Jordan type, although
we do note that
$\mc{X}_{M} = k \cdot y$ and $\mc{X}_{N} = k \cdot x$ in
this case.
\end{example}

These examples show that these modules are not closed under
taking direct summands because modules may decompose
into summands which are not of constant Jordan
type themseleves, but are symmetric in such a way that they
become
constant Jordan type when summed together.  Note that
this situation is contrasted with that of \cite{CFP} in
this respect.

We conclude this section with an example that shows that
there are indecomposable modules of stable Jordan type
other than 1, which even in simple cases is non-trivial.

\subsection{A Nontrivial Example} \label{SS: syzygy mod soc}
We begin by considering a proposition which will be useful
in constructing the example of interest in this section.

\begin{proposition} \label{P: syzygy restriction is surj}
Let $\mf{g}$ be a Lie superalgebra and let $\mf{h}$ be a
Lie subalgebra such that for each projective module $Q$
in $\F{g}$, $Q|_{\mf{h}}$ is projective in $\F{h}$.
Let $k \in \F{g}$ be the trivial module concentrated in
either degree and denote the decomposition by restriction
as $\Omega^n_{\mf{g}}(k)|_{\mf{h}} \cong
\Omega^n_{\mf{h}}(k) \oplus P_n$ for all $n \in \Z$ where
$P_n$ is projective in $\F{h}$.  Then there are extensions
\begin{align*}
0 \rightarrow P_n \rightarrow \Omega^n_{\mf{g}}(k) \rightarrow \Omega^n_{\mf{h}}(k) \rightarrow 0 \quad & \text{ for $n \geq 0$} \\
0 \rightarrow \Omega^n_{\mf{h}}(k) \rightarrow \Omega^n_{\mf{g}}(k) \rightarrow P_n \rightarrow 0 \quad & \text{ for $n \leq 0$}
\end{align*}
as modules in $\F{g}$.
\end{proposition}
\begin{proof}
We use an inductive argument as follows, noting that
when $n = 0$, we have $k \cong \Omega^0_{\mf{g}}(k)
\cong \Omega^0_{\mf{h}}(k)$ and so $P = 0$ and the claim
is trivial.  We proceed by induction and include the proof
for the $n \geq 0$ case and the $n \leq 0$ proof is dual to
the following.

Recall that $\Omega^{n+1}(k)$ is by definition
$\Omega^1(\Omega^n(k))$ and we have the following two
exact sequences which define $\Omega^{n+1}_{\mf{g}}(k)$ and
$\Omega^{n+1}_{\mf{h}}(k)$
\begin{equation*}
\begin{tikzpicture}[description/.style={fill=white,inner sep=2pt},baseline=(current  bounding  box.center)]

\matrix (m) [matrix of math nodes, row sep=2.5em,
column sep=2.5em, text height=1.5ex, text depth=0.25ex]
{		& & \Ker(\pi_{n+1}) & P_n & \\
	0	&	\Omega^{n+1}_{\mf{g}}(k)	&	P_{\mf{g}}
										&	\Omega^n_{\mf{g}}(k) & 0 \\
	0	&	\Omega^{n+1}_{\mf{h}}(k)	&	P_{\mf{h}}
										&	\Omega^n_{\mf{h}}(k) & 0 \\};

\path[->,font=\scriptsize]
	(m-2-1) edge (m-2-2)
	(m-2-2) edge (m-2-3)
	(m-2-3) edge node[auto] {$ \psi_\mf{g} $} (m-2-4)
	(m-2-4) edge (m-2-5)
	(m-3-1) edge (m-3-2)
	(m-3-2) edge (m-3-3)
	(m-3-3) edge node[auto] {$ \psi_\mf{h} $} (m-3-4)
	(m-3-4) edge (m-3-5); 
\path[->>, font=\scriptsize]
	(m-2-4) edge node[auto] {$ \pi_n $} (m-3-4);
\path[->>, dashed, font=\scriptsize]
	(m-2-3) edge node[auto] {$ \pi_{n+1} $} (m-3-3);
\path[right hook-latex, dashed, font=\scriptsize]
	(m-1-3) edge (m-2-3);
\path[right hook-latex, font=\scriptsize]
	(m-1-4) edge (m-2-4);
\end{tikzpicture}
\end{equation*}
where, both the surjection given by $\pi_n$ and its kernel,
exist by the inductive
hypothesis.  Note that this is in fact an isomorphism when
$n = 0$ and so $P_0 = 0$.
Additionally, the map $\pi_n \circ \psi_{\mf{g}}$ can be lifted
to a map $\pi_{n+1}$ since $\psi_{\mf{h}}$ is surjective.
Furthermore, $\pi_{n+1}$ can be seen to be surjective by
considering the commutative square restricted to $\mf{h}$
\begin{equation*}
\begin{tikzpicture}[description/.style={fill=white,inner sep=2pt},baseline=(current  bounding  box.center)]

\matrix (m) [matrix of math nodes, row sep=2.5em,
column sep=3.5em, text height=1.5ex, text depth=0.25ex]
{
P_{\mf{h}} \oplus P^{\perp} &	\Omega^n_{\mf{h}}(k) \oplus P_n \\
P_{\mf{h}} &	\Omega^n_{\mf{h}}(k) \\};

\path[->,font=\scriptsize]
	(m-1-1) edge node[auto] {$ \psi_\mf{h} \oplus \psi^{\perp} $} (m-1-2)
			edge node[auto] {$ \pi_{n+1} $} (m-2-1)
	(m-2-1) edge node[auto] {$ \psi_\mf{h} $} (m-2-2)
	(m-1-2) edge node[auto] {$ \pi_n $} (m-2-2); 
\end{tikzpicture}
\end{equation*}
and so $\pi_n$ is just projection onto $\Omega^n_{\mf{h}}(k)$,
$\pi_{n+1}$ is projection onto $P_{\mf{h}}$, and
$\psi^{\perp}$ gives a surjection of
$P^{\perp} = \Ker(\pi_{n+1})$ onto $P_n$ as $\mf{g}$-modules.

Since $\psi_{\mf{g}}(\Omega^{n+1}_{\mf{g}}(k)) = 0$, then
$\psi_{\mf{h}}(\pi_{n+1}(\Omega^{n+1}_{\mf{g}}(k))) = 0$ and
so $\pi_{n+1}(\Omega^{n+1}_{\mf{g}}(k)) \subseteq
\Omega^{n+1}_{\mf{h}}(k)$.  However, by again
recalling that $\pi_{n+1}$ is surjective and considering
the restriction to $\mf{h}$ above, we see that in fact
$\pi_{n+1}(\Omega^{n+1}_{\mf{g}}(k)) =
\Omega^{n+1}_{\mf{h}}(k)$ and that the kernel of this map
is given by the pullback of $\Omega^{n+1}_{\mf{g}}(k)$ and
$\Ker(\pi_{n+1})$ mapping into $P_{\mf{g}}$ and which we
denote $P_{n+1}$.  Thus the diagram can be completed to
\begin{equation*}
\begin{tikzpicture}[description/.style={fill=white,inner sep=2pt},baseline=(current  bounding  box.center)]

\matrix (m) [matrix of math nodes, row sep=2.5em,
column sep=2.5em, text height=1.5ex, text depth=0.25ex]
{	0	&	P_{n+1} & \Ker(\pi_{n+1}) & P_n & 0 \\
	0	&	\Omega^{n+1}_{\mf{g}}(k)	&	P_{\mf{g}}
										&	\Omega^n_{\mf{g}}(k) & 0 \\
	0	&	\Omega^{n+1}_{\mf{h}}(k)	&	P_{\mf{h}}
										&	\Omega^n_{\mf{h}}(k) & 0 \\};

\path[->,font=\scriptsize]
	(m-1-1) edge (m-1-2)
	(m-1-2) edge (m-1-3)
	(m-1-3) edge node[auto] {$ \psi^{\perp} $} (m-1-4)
	(m-1-4) edge (m-1-5)
	(m-2-1) edge (m-2-2)
	(m-2-2) edge (m-2-3)
	(m-2-3) edge node[auto] {$ \psi_\mf{g} $} (m-2-4)
	(m-2-4) edge (m-2-5)
	(m-3-1) edge (m-3-2)
	(m-3-2) edge (m-3-3)
	(m-3-3) edge node[auto] {$ \psi_\mf{h} $} (m-3-4)
	(m-3-4) edge (m-3-5); 
\path[->>, font=\scriptsize]
	(m-2-4) edge node[auto] {$ \pi_n $} (m-3-4)
	(m-2-3) edge node[auto] {$ \pi_{n+1} $} (m-3-3)
	(m-2-2) edge node[auto] {$ \pi_{n+1} $} (m-3-2);
\path[right hook-latex, font=\scriptsize]
	(m-1-2) edge (m-2-2)
	(m-1-3) edge (m-2-3)
	(m-1-4) edge (m-2-4);
\end{tikzpicture}
\end{equation*}
which yields the exact sequence
$0 \rightarrow P_{n+1} \rightarrow \Omega^{n+1}_{\mf{g}}(k) \rightarrow \Omega^{n+1}_{\mf{h}}(k) \rightarrow 0$
inductively from the previous one.  Note that $P_{n+1}$
is projective when restricted to $\F{h}$ since the other
two modules on the top row are as well and that
by construction,
$\Omega^{n+1}_{\mf{g}}(k)|_{\mf{h}} \cong
\Omega^{n+1}_{\mf{h}}(k) \oplus P_{n+1}$ as
desired.
\end{proof}

\begin{corollary} \label{C: syzygy restriction is surj}
Given the above assumptions with
the decomposition by restriction denoted
as $\Omega^n_{\mf{g}}(k)|_{\mf{h}} \cong
\Omega^n_{\mf{h}}(k) \oplus P_n$, then
\begin{align*}
\op{Hd}(\Omega^n_{\mf{h}}(k)) \subseteq
 \op{Hd}(\Omega^n_{\mf{g}}(k)) \quad & \text{ for $n \geq 0$} \\
\Soc(\Omega^n_{\mf{h}}(k)) \subseteq
 \Soc(\Omega^n_{\mf{g}}(k)) \quad & \text{ for $n \leq 0$}.
\end{align*}
\end{corollary}

Now we consider the case of
$\Omega^n(k) \in \F{g}$ where $\mf{g}$ is the
detecting subalgebra $\mf{f}_r = \sL \times \dots \times \sL$
and there are $r$ copies of $\sL$.
Then
$\Omega^n(k) / \Soc(\Omega^n(k))$ for $n > 0$ and
$\op{Rad}(\Omega^n(k))$ for $n < 0$ are of constant Jordan
type but are decomposable when
$r = 1$.  When $r > 1$ we show that for $n =1$ the module
is indecomposable and conjecture that this holds for all
$n \neq 0$.  Furthermore,
$\Omega^n(k) / \Soc(\Omega^n(k))$ and
$\op{Rad}(\Omega^{-n}(k))$ have the same Jordan type for
$n \geq 0$ and are in fact isomorphic since the projective
indecomposables in this block are self dual (up to parity
change).

This can be seen by considering
the projective covers of the simple
modules in the principal block, where indecomposable
modules of constant Jordan type necessarily exist,
shown later in Lemma \ref{L: cjt in prin block for f}.  Since the only simple
modules are $k_{ev}$ and $k_{od}$, and the induced modules
$P(0) = \ind{f} k$ have one dimensional
simple heads and simple socles,
these induced modules are the projective indecomoposable modules
in the block.

As in \cite[Section 5.2]{Talian-2013}, we observe that in the principal block,
we may simplify the situation from considering $U(\mf{f}_r)$-modules
to modules over a super symmetric algebra generated by
odd elements.  So let 
$V(\mf{a}) := S^{\text{sup}}(\mf{a}) \cong \Lambda(\mf{a})$ where
$\mf{a} = \od{(\mf{f}_r)}$ and has a fixed basis of
$\{x_1, \dots, x_r, y_1, \dots, y_r \}$.
This is because all even
elements commute and in the principal block, they act by zero,
allowing this reduction.  Let
$V(\mf{a}_s)$ denote an exterior algebra with $s$ elements
of degree 1.  Also note that under this equivalence we have
$$
\ind{f} k \cong V(\mf{a}_{2r}) \quad \text{ and } \quad
\mc{X} = \sum_{i = 1}^{r} a_i x_i + b_i y_i
$$
where at most one of the $a_i$ and $b_i$ are nonzero for
each $1 \leq i \leq r$.  However, an arbitrary element of
$V(\mf{a}_{2r})$ squares to zero, and given the results
of Section \ref{S: type f cjt}, it will be meaningful to
consider this example where $\mc{X}$ is given by
$$
\mc{X} = \sum_{i = 1}^{r} a_i x_i + b_i y_i
$$
where arbitrary coefficients are allowed.


Now we examine $\Omega^n(k) / \Soc(\Omega^n(k))$ when
$r > 1$ and $n > 0$.
Recall
that $\Omega^n(k)$ is of stable Jordan type $1[1]$ and
since we only have information up to stable equivalence,
the projective summands in the Jordan decompositions will
be generically denoted as $P_M$ for a module $M$.
In the decomposition 
$$
\Omega^n(k)|_{\gen{x}} \cong k \oplus P_{\Omega^n(k)}
$$
for $x \in \mc{X} \minuszero$,
the trivial summand must appear in the head of $\Omega^n(k)$
by Corollary \ref{C: syzygy restriction is surj}.

Since the socle of $\Omega^n(k)$ is contained in the socle
of the decomposition $M|_{\gen{x}} \cong k
\oplus P_M$ and
$k \cap \Soc(\Omega^n(k)) = \varnothing$, then
$\Soc(\Omega^n(k)) \subseteq \Soc(P_M)$.
Then taking $\Omega^n(k) / \Soc(\Omega^n(k))$ yields
$\dim(\Soc(\Omega^n(k)))$ new $k$ summands in each decomposition
of $M|_{\gen{x}}$ and so 
$\Omega^n(k) / \Soc(\Omega^n(k))$ is of stable Jordan type
$1 + \dim(\Soc(\Omega^n(k)))$.

In particular, for any $\mf{g}$ the projective cover
of the trivial module $k$ has a simple head and
a simple socle so $\Omega^1(k)$ has a simple socle and thus
$\Omega^1(k) / \Soc(\Omega^1(k))$ is a
module of stable Jordan type $2[1]$.  Similarly,
$\op{Rad}(\Omega^{-1}(k))$ has stable Jordan type $2[1]$.

Additionally, both these modules are indecomposable.
The argument for $\Omega^1(k) / \Soc(\Omega^1(k))$
is given here and the one for $\op{Rad}(\Omega^{-1}(k))$
is dual.  Assume that $\Omega^1(k) / \Soc(\Omega^1(k))$
decomposes into $M \oplus N$.  The head of
$\Omega^1(k) / \Soc(\Omega^1(k))$ is isomorphic to
$\Lambda^1(\mf{a})$ and the socle is isomorphic to
$\Lambda^{\dim \mf{a} -1}(\mf{a})$ and these two are
distinct since $r > 1$ by assumption.  Then we can write
$\Hd(\Omega^1(k) / \Soc(\Omega^1(k))) \cong \Hd(M) \oplus
\Hd(N)$ with bases $\{ m_1, \dots, m_s\}$ and
$\{n_1, \dots, n_t\}$ such that their union is a basis
for $\Hd(\Omega^1(k) / \Soc(\Omega^1(k)))$ (and so
$s + t = 2r$).  Then since $m_1 \otimes n_1 = - n_1 \otimes m_1
\neq 0$
and is in the image of both $\Hd(M)$ and $\Hd(N)$, then
they have nontrivial intersection unless either $s$ or $t$
is 0.  Thus, $\Hd(\Omega^1(k) / \Soc(\Omega^1(k)))$ is
indecomposable.

It is possible that $\Omega^n(k) / \Soc(\Omega^n(k))$
will be indecomposable for all $n > 0$ but the
techniques used here are not sufficient to provide a general
proof.

\section{Endotrivial Modules} \label{S: endotrivial} \label{S: endotriv}
Endotrivial modules are an important and interesting class
of module which has been studied in various contexts since
Dade introduced the notion in 1978 in modular representation
theory (\cite{Dade1-1978}, \cite{Dade2-1978}).
Other significant progress includes \cite{CT-2004}, \cite{CT-2005},
and \cite{Puig-1990}
for representations of $p$-groups, \cite{CN-2009}, \cite{CN-2011}
for finite group schemes, and \cite{CMN-2014}, \cite{CMN-2006} for finite
groups of Lie type. 
In general, if $T(G)$ denotes the set of endotrivial $G$-modules
in the stable module category (in one of the above contexts), then
such modules form a group where
the operation is the tensor product.

The author began the study of endotrivial modules
for Lie superalgebras in \cite{Talian-2013} and
gave further study
in \cite{Talian-2015}.  These papers present classifications
of the group of endotrivial modules for detecting
subalgebras (as defined in \cite{BKN1-2006}) and for
$\mf{gl}(m|n)$, respectively.

The following theorem is
powerful in that it allows us
to determine when a module is endotrivial by checking
this property locally, and
recalling the remark at the end of Section
\ref{S: Properties of CJT}, this
must be done only at a finite number of points in many cases.
The theorem is analogous to \cite[Theorem 5.6]{CFP}
and is proven in a similar way. 
\begin{theorem} \label{T: endo are CJT 1}
Let $\mf{g}$ be a finite dimensional Lie superalgebra such
that $\mc{X}$ spans $\od{\mf{g}}$ and
$M \in \F{g}$ be a module.  Then $M$ is an endotrivial
module if and only if $M$ a module
of constant Jordan type $1[1] + m[2]$ for some $m \geq 0$.
\end{theorem}
\begin{proof}
First assume that $M$ is an endotrivial module.  Then
by definition,
$$
M \otimes M^* \cong \End_k(M) \cong  k_{ev} \oplus P
$$
and since restriction commutes with the above,
$$
M|_{\gen{x}} \otimes M^*|_{\gen{x}} \cong \End_k(M|_{\gen{x}}) \cong \End_k(M)|_{\gen{x}} \cong
k_{ev}|_{\gen{x}} \oplus P|_{\gen{x}}
$$
for any $x \in \mc{X} \minuszero$.

By considering the Jordan type of $M$ at $x$,
$M|_{\gen{x}} \cong k^{\oplus a_1}
\oplus Q^{\oplus a_2}$, and recalling Proposition
\ref{P: closure of CJT} (applied at the single point $x$),
\begin{equation} \label{E: M otimes M*}
M|_{\gen{x}} \otimes M^*|_{\gen{x}} \cong
k^{\oplus a_1 \cdot a_1}
\oplus Q^{\oplus 2(a_1 \cdot a_2 + a_2^2)}
\end{equation}
it is clear that if $M$ has stable Jordan type $a_1[1]$ at $x$ then
$\End_k(M)$ has stable Jordan type $a_1^2[1]$ at $x$.  Since
$\End_k(M)$ has stable Jordan type $1[1]$ at all points $x \in \mc{X} \minuszero$,
then $M$ is of stable Jordan type $1[1]$.

Now assume that $M$ is of stable Jordan type $1[1]$ and
consider the endomorphism algebra of $M$.  If we fix
a basis for $M$, we can think of any endomorphism as
a $d \times d$ matrix where $d = \dim(M)$ and the trace
of the endomorphism is independent of the choice of basis.
Since the field $k$ has characteristic 0, there is a homogeneous
degree $\ov{0}$ map
\begin{align*}
\alpha: k_{ev} & \rightarrow \End_k(M) \\ c & \mapsto \dfrac{c}{d}\cdot \op{Id}_M
\end{align*}
and the composition $\op{Tr} \circ \alpha = \op{Id}_k$ and
so there is a splitting
$\End_k(M) \cong k_{ev} \oplus \Ker(\op{Tr})$ as
modules in $\mc{F}$.

Since $M$ has stable Jordan type $1[1]$, considering
Equation~\ref{E: M otimes M*} it is clear that $\End_k(M)$ also
has stable Jordan type $1[1]$ as well and in particular,
$\Ker(\op{Tr})|_{\gen{x}}$ is projective for all $x \in \mc{X} \minuszero$
by comparing the projective summands in the two decompositions.
Finally, by recalling that $\mc{X}$ detects projectivity
(\cite[Theorem 3.4]{DS-2005}), since
$\Ker(\op{Tr})|_{\gen{x}}$ is projective for all $x$,
$\Ker(\op{Tr})$ is projective as a module in $\mc{F}$.
Thus, $\End_k(M)$ is
endotrivial as claimed.
\end{proof}

The following corollary follows immediately from
Proposition \ref{P: CJT has max atyp}.

\begin{corollary} \label{C: endo in maximal atyp}
Let $\mf{g}$ be simple basic classical and $M \in \F{g}$
be an endotrivial module.  Then
$\mc{X}_M = \mc{X}$ and $M$ lies in a block of maximal
atypicality.
\end{corollary}

This partially recovers an important reduction used
in the classification
of endotrivial modules for detecting subalgebras in
\cite[Lemma 5.2]{Talian-2013} since
$\sL$ is a simple basic
classical Lie superalgebra.  Additionally,
in the proof of that classification, certain conditions are
used to describe endotrivial modules and we note here that
\cite[Section 5.3, condition (2)]{Talian-2013} identifies
a module as stable Jordan type $1[1]$.

This corollary can be recovered by the Kac-Wakimoto conjecture
in the cases where it is known to hold, i.e. $\mf{gl}(m|n)$ and
$\mf{osp}(m|2n)$ by \cite{BKN2-2009},
\cite{Serg}, and \cite{Kujawa}.
The conjecture states that for a simple module $L$,
$$
\sdim(L) \neq 0 \quad \text{if and only if} \quad
\op{atyp}(L) =
\op{def}(\mf{g}).
$$

\begin{proposition}
Let $\mf{g}$ be a Lie superalgebra where the Kac-Wakimoto
conjecture holds (either $\mf{gl}(m|n)$ or
$\mf{osp}(m|2n)$) and let $\mc{E}_{\mf{g}}$ be the category of finite
dimensional integrable $\mf{g}$-supermodules, a full
subcategory of $\F{g}$.
Let $M \in \mc{E}_{\mf{g}}$
be an endotrivial module.  Then
$\op{atyp}(M) = \op{def}(\mf{g})$, i.e. $M$ lies in a block of maximal
atypicality.
\end{proposition}
\begin{proof}
A corollary
of Proposition \ref{P: closure of CJT} is that
$\sdim(M) \neq 0$ if and only if
$\sdim (M \otimes M^*) \neq 0$ and
so if $M$ is an indecomposable
endotrivial module, then $M \otimes M^* \cong k \oplus P$
and since $\sdim(P) = 0$, then we conclude that
$\sdim(M) \neq 0$ and thus $\op{atyp}(M) = \op{def}(\mf{g})$
by the Kac-Wakimoto conjecture.
\end{proof}

\section{Constant Jordan Type for Type $\mf{f}$ Detecting Subalgebras} \label{S: type f cjt}
Recall the detecting subalgebra $\mf{f}_r = \sL \times \dots \times \sL$
first mentioned in Section \ref{SS: syzygy mod soc}.
In this setting, a natural
enlargement of $\mc{X}$ presents itself and allows for more
progress to be made in understanding modules of constant Jordan type.
In particular, we are able to achieve closure of modules
of constant Jordan type under taking direct summands (Proposition
\ref{P: cjt closed under sums}), construct (super) vector bundles
on $\P^{2r-1}$ (Section \ref{SS: SVB}) and completely classify
modules of constant Jordan type over $\mf{f}_1$ (Theorem \ref{T: cjt class for f1}).

\subsection{Enlarging the Self Commuting Cone}

Since this Lie superalgebra is a product (when $r > 1$) it is not
not simple basic classical and so Proposition \ref{P: CJT has max atyp}
does not apply.  Thus, we
begin by showing that for this particular case of interest,
non-projective 
indecomposable modules of constant Jordan type must
necessarily exist in the principal block, i.e., the weights
of $M$ are all zero.


\begin{lemma} \label{L: cjt in prin block for f}
Let $M$ be a non-projective indecomposable $\mf{f}_r$-module
of constant Jordan type.
Then if $\ev{\sL}$ has basis $\{t_1, \dots, t_r \}$ and
$\od{\sL}$ has basis $\{x_1, \dots, x_r, y_1, \dots, y_r \}$
such that $[x_i,y_i] = t_i$, then $t_i.m = 0$ for all
$m \in M$.
\end{lemma}
\begin{proof}
Since $M$ has constant Jordan type with $a_1 > 0$, then
let $v$ be a generator for one of the trivial summands in the
decomposition
$M|_{\gen{x_i}} \cong k^{\oplus a_1} \oplus P^{\oplus a_2}$, so
that $x_i.v = 0$.  Then consider $x.y.v = y.x.v + [x,y].v =
t_i.v$.  If $y.v = 0$, then we have $0 = t_i.v$ and if
$w = y.v \neq 0$ then $x.w = t_i.v$, but since $v$ is a trivial
summand as an $\gen{x_i}$-modules, then it cannot be in
the image of a nonzero vector and thus $t_i.v = 0$ again.
Since this is true for all $1 \leq i \leq r$, and $M$
is indecomposable, we conclude that $t_i.m = 0$ for all
$1 \leq i \leq r$ and $m \in M$.
\end{proof}

We now make the same reduction used in \ref{SS: syzygy mod soc}
since we now know any module of constant Jordan type is
in the principal block.  As before, let
$V(\mf{a}_s)$ denote an exterior algebra with $s$ elements
generated by $\gen{1, x_1, \dots, x_s}$ as an algebra and we
consider representations of $U(\mf{f}_r)$ in the principal
block as representations of $V(\mf{a}_{2r})$.  Furthermore,
recall that each of the $a_i$ are odd and thus act on a
supermodule $M$ via an odd endomorphism.

Because of this general reduction, we can say more about
$\mf{f}$-modules of constant Jordan type if we
require that $\cjt{x}$ for all $x \in \od{\mf{f}}$ which is
the self commuting cone $\mc{X}$ associated to $V(\mf{a})$.

We refer to $\U{f}$-modules which satisfy
$\cjt{x}$ for all $x \in \od{\mf{f}}$ as $\mf{f}$-modules of
\emph{strong constant Jordan type} when needed.  For simplicity
and full generality, much of the following is treated in the
context of $V(\mf{a})$-supermodules and for such modules,
the term constant
Jordan type automatically implies the stronger condition under the
correspondence between $\mf{f}$ and $\mf{a}$.

\begin{proposition}[Benson] \label{P: cjt closed under sums}
Let $M$ and $N$ be $V(\mf{a})$-modules.  Then $M$ and $N$
both have constant Jordan type if and only if $M \oplus N$ has
constant Jordan type.
\end{proposition}
\begin{proof}
The proof is the same as in \cite[Theorem 4.1.9]{Benson-Book}
as the result is based on \cite[Theorem 3.6.3]{Benson-Book}
which holds for exterior algebras.
\end{proof}

\begin{remark}
This generalization eliminates the counter examples
presented in \ref{SS: examples}.  This can be seen by
considering each of the modules
presented in Examples~\ref{Ex: 1}, \ref{Ex: 2},
and \ref{Ex: 3} restricted to $x+y$.  In these cases,
the Jordan type at $x+y$ is $2[2]$, $2[1] + 2[2]$, and
$4[2]$ respectively and so none of the examples are of
constant Jordan type when $\mc{X}$ is enlarged to all of
$\od{\mf{f}}$.
\end{remark}

\subsection{Super Vector Bundles}
An important application of modules of constant Jordan type
is their use in construction vector bundles over
projective space $\P^n$ given by a
construction first introduced in \cite{FP-1}
and detailed in \cite[Section 7]{Benson-Book}.  In keeping with the
convention in \cite{Benson-Book}, we take the term super vector
bundle to mean locally free sheaf of finite super rank $(m|n)$,
i.e. a sheaf of $\mc{O}$-supermodules which is locally free and has
$m$ even basis elements and $n$ odd basis elements.

The setting of $V(\mf{a})$-modules of constant Jordan type
provides two different
analogous construction which are considered
here.  The first possibility is to use such modules to find super vector
bundles over a super manifold, which quickly fails.
The other
is accomplished
by adapting the construction in \cite[Section 7]{Benson-Book}
to fit the context of $V(\mf{a})$-modules.

\subsubsection{Super Manifolds} \label{SS: super man}
First we construct a natural
super $k$-manifold associated to $V(\mf{a}_s)$
over which the super vector bundles should lie.
Note that $V(\mf{a}_s)$ already is isomorphic
to an odd coordinate system of a splitting neighborhood
(in the terminology of \cite{Rogers}).  Thus, we can view $V(\mf{a}_s)$ as
a sheaf of super commutative $k$-algebras
over single pointed super $k$-manifold.  Then
the pair
$$
\mc{M} := (\Spec V(\mf{a}_s), V(\mf{a}_s))
$$
is a super $k$-manifold of dimension $(0|s)$, since $V(\mf{a}_s)$ is nilpotent
and thus has trivial spectrum, or in other words, $\mc{M}$ is
a super manifold consisting of
a single point and nilpotent ``fuzz.''

At this point, the construction has effectively failed since in the
super manifold setting, the vector spaces obtained by considering
the fibers $M_x = \Ker(x)/ \Image(x)$ for each $x \in V(\mf{a}_s)$
are now concentrated topologically over one point. 

\subsubsection{Super Vector Bundles over Projective Space} \label{SS: SVB}

Instead, we now consider
extending the construction detailed in \cite[Section 7]{Benson-Book}
which is particularly
useful for the setting of $V(\mf{a})$-modules.  First,
a slight generalization of of the algebraic vector bundles considered
in \cite{Benson-Book} must be defined.
\begin{definition}
Let $X$ be a connected
reduced Noetherian scheme with structure sheaf
$\mc{O}_X$.  Let $\ms{F}$ be a sheaf of $\mc{O}_X$ modules which
is locally free of rank $a$.
Then $\ms{F}$ is called a \emph{vector bundle} of rank $a$.

Furthermore, if	 a sheaf $\ms{F}$ of $\Z_2$ graded modules and
is locally free of rank $(r_{ev}|r_{od})$, where there are $r_{ev}$ even
basis elements and $r_{od}$ odd ones,
then we say $\ms{F}$ is a \emph{super vector bundle}
of rank $(r_{ev}|r_{od})$.
\end{definition}

Note that since we are working with standard schemes, we will be
consider the structure sheaf to be concentrated in the even degree
for the purposes of introducing sheaves of $\Z_2$ graded $\mc{O}_X$-modules.
This means that $\mc{O}_X(U)$ will act by an even endomorphism on
$\mc{F}(U)$ for any open $U \subseteq X$.

The vector
bundles are constructed as follows.
Let $V(\mf{a}_s)$ be as above.  Then the Jacobson radical $$J = J(V(\mf{a}_s))
= \bigoplus_{i > 0} \Lambda^i(\mf{a}_s)$$
is generated by $\gen{x_1, \dots, x_s}$ and $J /J^2$ has
a basis of $\{\ov{x}_1, \dots, \ov{x}_s \}$ where $\ov{x}_i$
denotes the image of $x_i$ in $J / J^2$.  For
$$
\alpha = (\lambda_1, \dots, \lambda_s) \in \A^s \setminus \{0\}
$$
define
$$
x_\alpha = \lambda_1 x_1 + \dots \lambda_s x_s \in J
$$
which satisfies $x_\alpha^2 = 0$ by construction.

Recall that the cohomology ring for the superalgebra
$V(\mf{a}_s)$, whose construction and subsequent
computation is given in
\cite[Theorem 2.5.2]{BKN1-2006} as
$$
\H^\bullet(\mf{a}_s, \ev{(\mf{a}_s)}; k) \cong S(\mf{a}_s^*)^{\ev{(\mf{a}_s)}}.
$$
There are no invariants since there is no even component
of $\mf{a}_s$ and so the cohomology
is given by symmetric functions on the odd generators of
the Lie superalgebra $\mf{a}_s$, which in this case is the whole
superalgebra.

Thus there is an isomorphism
$S(\mf{a}_s^*) \cong k[Y_1, \dots, Y_s]$
where the $Y_i$ are linear functions defined by
$Y_i(x_j) = \delta_{ij}$.
We also observe that this can be thought of as
the coordinate ring of an affine space of dimension $s$ with
basis elements $x_i$.  We denote this space by
$\A^s$ and let $\P^{s-1}$ denote the associated
projective space with corresponding
structure sheaf $\mc{O}$ and twists $\mc{O}(j)$.  For a
$V(\mf{a}_s)$ module $M$, let $\tilde{M}$ denote the
super vector bundle $M \otimes \mc{O}$ and $\tilde{M}(j) =
M \otimes \mc{O}(j)$ for the $j$th twist of the sheaf
$\tilde{M}$.

Continuing to follow \cite{Benson-Book}, we define
$\theta_M : \tilde{M}(j) \rightarrow \tilde{M}(j+1)$ by
$$
\theta_M(m \otimes f) = \sum_{i = 1}^s x_i.m \otimes Y_i f
$$
for all $j \in \Z$.  Note that this is an odd morphism of
sheaves of supermodules because if $m$ is homogeneous,
$x_i.m$ has opposite parity from $m$ for all $i$ and
$Y_i f$ is necessarily even since $\mc{O}$ is by definition.

As shown in \cite[Section 7.3]{Benson-Book}, we can
identify the fibers of $\mc{O}(j)$ at $\ov{\alpha}$
with $k$ (concentrated in even degree since $\mc{O}$ is even)
by fixing a choice of $\alpha$ lying over
$\ov{\alpha}$.  Additionally, since
$\tilde{M}(j) = M \otimes \mc{O}(j)$, this choice of 
$\alpha$ gives an identification of the fiber of $\tilde{M}(j)$
at $\ov{\alpha}$ with $M$ and subsequently,
the action of $\theta_M$ on the fiber at a point $\ov{\alpha}$ is
given by multiplication by $x_\alpha$, detailed as follows for
the case of $\mc{O}$ (see \cite{Benson-Book} for the slightly more
general case of $\mc{O}(j)$).

For a point $\ov{\alpha} \in \P^{s-1}$, consider an affine patch
containing this point, $\mc{O}(U_\ell)$ where the $\ell$th coordinate
does not vanish.  If we make the identification
$$
\mc{O}(U_\ell) \cong k[U_\ell] \cong k[Y_1Y\inv_\ell, \dots , \widehat{Y_\ell Y\inv_\ell},
\dots, Y_s Y\inv_\ell],
$$
then the fiber of $\mc{O}$ at $\ov{\alpha}$ is given by
$$
k[U_\ell] \otimes_{k[U_\ell]} k \cong k
$$
where $k$ is a $k[U_\ell]$-module by evaluation at $x_{\ov{\alpha}}$.  Then
the fiber of $M \otimes \mc{O}$ at $\ov{\alpha}$ is isomprhic to
$$
M \otimes k[U_\ell] \otimes_{k[U_\ell]} k \cong M \otimes k \cong M
$$
and so the action of $\theta_M$ on this fiber sends
$m \otimes 1 \otimes_{k[U_\ell]} 1$ to
\begin{gather*}
\sum_i x_i .m \otimes Y_i \otimes_{k[U_\ell]} 1 = \sum_i x_i .m \otimes 1 \otimes_{k[U_\ell]} \lambda_i \\
= \sum_i \lambda_i x_i.m \otimes 1 \otimes_{k[U_\ell]} 1 = x_{\alpha}.m \otimes 1 \otimes_{k[U_\ell]} 1
\end{gather*}
for some choice of $\alpha$ lying over $\ov{\alpha}$.  Since the image
and kernel of multiplication by $x_\alpha$ on $M$ are
invariant under scaling and hence the choice of $\alpha$,
we can use this observation
to construct some interesting functors.
Furthermore,
this computation implies that $\theta_M^2 \equiv 0$ since $x_\alpha^2 = 0$ for
any $\alpha \in \A^s \setminus \{ 0 \}$.

Continuing, we define functors $\ms{F}_i$ for $i=1$ or 2 from
from $V(\mf{a}_s)$-supermodules to
coherent sheaves of supermodules on $\P^{s-1}$ by
$$
\ms{F}_i(M) = \dfrac{\Ker \theta_M \cap \Image \theta_M^{i-1}}{\Image \theta_M^i}.
$$
Note that for $i = 1$ this gives the functor
$$
\ms{F}_1(M) = \dfrac{\Ker \theta_M}{\Image \theta_M}
$$
and when $i = 2$ this is simply
$$
\ms{F}_2(M) = \Image \theta_M
$$
since $\Image \theta_M \subseteq \Ker \theta_M$ and $\theta_M^2 \equiv 0$
as a map of sheaves of supermodules.  This is also why we do not define
$\ms{F}_i$ for $i > 2$ because the corresponding generalization
gives sheaves which are identically
zero and thus of no interest.

For $\alpha \in \A^s \setminus \{0\}$
with residue field $k(\ov{\alpha})$ where $\ov{\alpha}$ is
the image of $\alpha$ in $\P^{s-1}$,
the corresponding construction for specialization is given
by defining maps
\begin{gather*}
x_\alpha: M \otimes k(\ov{\alpha})  \rightarrow M \otimes k(\ov{\alpha}) \\
m \otimes v  \mapsto x_\alpha.m \otimes v
\end{gather*}
for each $\alpha \in \A \setminus \{0\}$ and a functor
from $V(\mf{a}_s)$-supermodules to super vector spaces
$$
\ms{F}_{i, \alpha}(M) = \dfrac{\Ker x_\alpha \cap \Image x_\alpha^{i-1}}{\Image x_\alpha^i}
$$
which only depends on the image
$\ov{\alpha} \in \P^{s-1}$ by construction. 
Note that $\ms{F}_{1, \alpha}(M) \cong M_{x_\alpha}$ as super vector spaces.
These functors are particularly interesting in relation to
modules of constant Jordan type in which case the
resulting sheaves are actually (super) vector bundles.

The proof of the main theorem (Theorem \ref{T: vector bundles})
relies on another theorem analogous
to that of \cite[Theorem 5.2.2]{Benson-Book} but
for the case of super vector bundles over projective space.

\begin{theorem} \label{T: vector bundles theorem}
Let $\P^{s-1}$ denote the projectivization of affine $s$ space
$\A^s$ with structure sheaf $\mc{O}$.
\begin{enumerate}
\item If $\ms{F}$ is the coherent sheaf of
$\mc{O}$-supermodules of, then
the following are equivalent. \label{claim1}
\begin{enumerate}
\item The sheaf $\ms{F}$ is a super vector bundle of rank
$(r_{ev}|r_{od})$. \label{pp1}

\item The even and odd dimensions of the fiber
$\dim_{k(\ov{\alpha})} \ms{F}_{\ov{\alpha}}
\otimes_{\mc{O}_{\ov{\alpha}}} k(\ov{\alpha})$ is constant
for all $\ov{\alpha} \in \P^{s-1}$. \label{pp2}
\end{enumerate}

\item If $f: \ms{F} \rightarrow \ms{F}'$ is an odd map of super
vector bundles (where $\ms{F}'$ has rank
$(r_{ev}|r_{od})$)
on $\P^{s-1}$ then the following are equivalent. \label{claim2}
\begin{enumerate}
\item The cokernel of $f$ is a super vector bundle of rank
$(r_{ev} - r'_{od}| r_{od} - r'_{ev})$. \label{qq1}

\item The induced map of fibers
$$
\ov{f} : \ms{F}_{\ov{\alpha}} \otimes_{\mc{O}_{\ov{\alpha}}} k(\ov{\alpha}) \rightarrow \ms{F}'_{\ov{\alpha}} \otimes_{\mc{O}_{\ov{\alpha}}} k(\ov{\alpha})
$$
has constant rank $(r'_{ev} | r'_{od})$ for all $\ov{\alpha} \in \P^{s-1}$.
\label{qq2}
\end{enumerate}
If these hold, then the image of $f$ is a super vector bundle
of rank $(r'_{ev}|r'_{od})$.
\end{enumerate}
\end{theorem}
\begin{proof}
We proceed similarly to \cite{Benson-Book}.  Define a function
$\phi: \P^{s-1} \rightarrow \Z \times \Z$ by assigning
to each point $\ov{\alpha} \in \P^{s-1}$ the pair
$$
(\dim_{k(\ov{\alpha})} \ev{(\ms{F}_{\ov{\alpha}} \otimes_{\mc{O}_{\ov{\alpha}}} k(\ov{\alpha}))}, \dim_{k(\ov{\alpha})} \od{(\ms{F}_{\ov{\alpha}}  \otimes_{\mc{O}_{\ov{\alpha}}} k(\ov{\alpha})))}).
$$
For any coherent sheaf $\ms{F}$ and for any
$m,n \in \Z$, the
sets
$$\P^{s-1}_{\ov{0} < m} = \{ \ov{\alpha} \in \P^{s-1} \st
\dim_{k(\ov{\alpha})} \ev{(\ms{F}_{\ov{\alpha}}  \otimes_{\mc{O}_{\ov{\alpha}}} k(\ov{\alpha})))} < m  \}  $$
$$\P^{s-1}_{\ov{1} < n} = \{ \ov{\alpha} \in \P^{s-1} \st
\dim_{k(\ov{\alpha})} \od{(\ms{F}_{\ov{\alpha}}  \otimes_{\mc{O}_{\ov{\alpha}}} k(\ov{\alpha})))} < n \}  $$
are open, seen as follows.

It is sufficient to show the claim when over an affine variety
$X = \Spec R$ where $R$ is Noetherian and $\ms{F} = \tilde{M}$,
the sheaf associated to a finitely generated $R$-supermodule $M$.
Finally, recall that $R$ acts evenly on $\ms{F}$, as this key fact will be
used repeatedly without further comment.

If $\ov{\alpha} \in X \cap \P^{s-1}_{\0 < m}$
has corresponding prime ideal
$\mf{p}$ of $R$, then
$\ms{F}_{\ov{\alpha}} \otimes_{\mc{O}_{\ov{\alpha}}} k(\ov{\alpha})
\cong M_{\mf{p}}/ \mf{p}M_{\mf{p}}$.
Let $m' = \dim_{k(\ov{\alpha})} \ev{(M_{\mf{p}}/ \mf{p}M_{\mf{p}})}$ so $m' < m$ by definition.
Let
$\tilde{v}_1, \dots, \tilde{v}_{m'}$
and $\tilde{w}_1, \dots, \tilde{w}_{n'}$
be a homogeneous set of elements (concentrated
in even
and odd degree, respectively)
of $M_{\mf{p}}$ such that
the images
$\ov{v}_1, \dots, \ov{v}_{m'}$
in $\ev{(M_{\mf{p}}/ \mf{p}M_{\mf{p}})}$ form a basis
for this space.  By Nakayama's lemma,
$\tilde{v}_1, \dots, \tilde{v}_{m'}$
generate $\ev{(M_{\mf{p}})}$ and since there are
a finite number of the $v_i$, by clearing the denominators we
can assume that these are the images of a set of homogeneous
even elements
$v_1, \dots, v_{m'} \in M$.

If $y_1, \dots, y_{d_{ev}}, z_1, \dots, z_{d_{od}}$ is
a homogeneous basis for $M$.  The image of each $y_i$ 
in $M_{\mf{p}}$ is a linear combination of the $\tilde{v}_i$
with coefficients in $R_{\mf{p}}$.  Again,
each of these linear combination has only a finite number of
elements so we can clear the denominators by multiplying by
a single homogeneous even
element $r \in R$ with $r \notin \mf{p}$
so that each of the $ry_i$
are linear combinations of the $v_i$
with coefficients in $R$.

Then define $U_r$ to be the open set consisting of prime ideals
$\mf{q}$ such that the for the fixed $r$ above,
$r \notin \mf{q}$.  Then for each $\mf{q} \in U_r$, we can
clear the denominators as described above and write the
images of the $y_i$
as linear combinations of the
$v_i$
with coefficients in $R_{\mf{q}}$
which shows that $M_{\mf{q}}$ is generated at most by
$m'$ even elements
and that
$U_r \subseteq X \cap \P^{s-1}_{\0 < m}$.  Thus, we have shown that
$X_i \cap \P^{s-1}_{\0 < m}$ for each $X_i$ in an open affine
covering of $\P^{s-1}$, hence $\P^{s-1}_{\0 < m}$ is open and similarly
$\P^{s-1}_{\1 < n}$ is as well.  With this fact established,
we proceed to the claims of the theorem.

(\ref{pp1}) $\Rightarrow$ (\ref{pp2}) Let $(S)^c$
denote the compliment of a set $S$.
By assumption, for each $\ov{\alpha} \in \P^{s-1}$ there
is an open neighborhood of $\ov{\alpha}$ on which $\ms{F}$
is free and thus, $\phi$ is constant.  By composing $\phi$
with projections onto the separate factors, we get maps
$\phi_{\0} = \pi_{\0} \circ \phi$ and
$\phi_{\1} = \pi_{\1} \circ \phi$
which are constant as well.  Thus, there are open neighborhoods
around each point of the sets
$(\P^{s-1}_{\ov{0} < m})^c$ and
$(\P^{s-1}_{\ov{1} < n})^c$ for each $m,n \in \Z$ such that
these sets are nonempty.  Therefore these
sets are open as well and since $\P^{s-1}$ is connected, each of
these nonempty sets must be all of $\P^{s-1}$ and so
$\phi_{\0}$ and $\phi_{\1}$ are constant which yields that
$\phi$ is constant as well.

(\ref{pp2}) $\Rightarrow$ (\ref{pp1}) Now we assume that
$\phi$ as defined above is constant on $\P^{s-1}$.  Again it
suffices to show that $\mc{F}$ is locally free of
rank $\phi(\ov{\alpha})$ on an open affine set since this
is a local condition.  Thus let $X = \Spec R$ for some
Noetherian $R$ and we assume that $\ms{F} = \tilde{M}$ for
some finitely generated supermodule $M$ where a point
$\ov{\alpha}$ corresponds to $\mf{p}$.  Let $v_i, w_j$,
$\tilde{v}_i, \tilde{w}_j$
and $y_i, z_j$ be as above and so
again by clearing denominators, there is some $r \in R$ such
that $r \notin \mf{p}$ so that the images of
$ry_i$ and $rz_j$ in $M_{\mf{p}}$ are
linear combinations of the $\tilde{v}_i$ and $\tilde{w}_j$.  If $R_r$ and $M_r$ denote $R$ and $M$ with $r$ inverted,
then the images of the $x_i$ and $y_j$ in $M_r$ are linear
combinations of the images of the $v_i$ and $w_j$ respectively
and thus
$M_r$ is generated by the images of $v_i$ and $w_j$.  Since the
$v_i$ are indexed from $1, \dots, \phi_{\0}$ and the
$w_i$ from $1, \dots, \phi_{\1}$ we obtain a surjective
homomorphism and corresponding exact sequence
$$
\begin{tikzpicture}[start chain] {
	\node[on chain] {$0$};
	\node[on chain] {$K$} ;
	\node[on chain] {$R_r^{\phi_{\0}(\ov{\alpha}) + \phi_{\1}(\ov{\alpha})}$};
	\node[on chain] {$M_r$};
	\node[on chain] {$0$}; }
\end{tikzpicture}
$$
where the kernel is denoted by $K$.  We can apply the (exact)
functor of localization to obtain the sequence
$$
\begin{tikzpicture}[start chain] {
	\node[on chain] {$0$};
	\node[on chain] {$K_{\mf{q}}$} ;
	\node[on chain] {$R_{\mf{q}}^{\phi_{\0}(\ov{\alpha}) + \phi_{\1}(\ov{\alpha})}$};
	\node[on chain] {$M_{\mf{q}}$};
	\node[on chain] {$0$}; }
\end{tikzpicture}
$$
for each prime ideal $r \notin \mf{q}$.  We assumed $\phi$ was constant, so
$\ev{(M_{\mf{q}}/\mf{q}M_{\mf{q}})}$ has dimension $\ev{\phi}(\ov{\alpha})$
and $\od{(M_{\mf{q}}/\mf{q}M_{\mf{q}})}$ has dimension $\od{\phi}(\ov{\alpha})$
over $R_{\mf{q}}/ \mf{q}R_{\mf{q}}$ and so $K_{\mf{q}} \subseteq \mf{q} R_{\mf{q}}^{\phi_{\0}(\ov{\alpha}) + \phi_{\1}(\ov{\alpha})}$.  Thus,
the coordinates of an element in $K$ are in $\mf{q}$.  Since $X \subseteq \P$
is reduced, $R_r$ has no nilpotent elements, so the intersection of the
prime ideals of $R_r$ is zero and then so is $K$.  Therefore, $\ms{F}$
is free on the open subset 
of $X$ defined by $r$.

The proof of part (\ref{claim2}) is the same as in \cite{Benson-Book} with
attention given to the even and odd components of the vector bundles.
\end{proof}

\begin{theorem} \label{T: vector bundles}
Let $M$ be a $V(\mf{a}_s)$-module of constant Jordan type
$(a_{ev}|a_{od})[1] + a_2[2]$.  Then
\begin{enumerate}
\item $\ms{F}_1(M)$ is a super vector bundle of rank 
$(a_{od}|a_{ev})$  over $\P^{s-1}$ with fiber over $\ov{\alpha}$ isomorphic to $M_{x_\alpha}$; \label{tt1}

\item $\ms{F}_2(M)$ is a vector bundle of rank $a_2$ over $\P^{s-1}$ with fiber
over $\ov{\alpha}$ isomorphic
to $\Soc(P|_{\gen{x_\alpha}})$ where $\cjt{x_\alpha}$; \label{tt2}
\end{enumerate}

Furthermore, if $f : M \rightarrow N$ is a homogeneous
map of (super) modules of
constant Jordan type, then for any $\ov{\alpha} \in \P^{s-1}$
with residue field $k(\ov{\alpha})$ there is a diagram
\begin{equation*}
\begin{tikzpicture}[description/.style={fill=white,inner sep=2pt},baseline=(current  bounding  box.center)]

\matrix (m) [matrix of math nodes, row sep=2.5em,
column sep=3em, text height=1.5ex, text depth=0.25ex]
{
\ms{F}_i(M) \otimes_{\mc{O}} k(\ov{\alpha}) & \ms{F}_i(N) \otimes_{\mc{O}} k(\ov{\alpha}) \\
\ms{F}_{i, \alpha}(M)  &	\ms{F}_{i, \alpha}(N)  \\};

\path[->,font=\scriptsize]
	(m-1-1) edge node[auto] {$\ms{F}_i(f) $} (m-1-2)
			edge node[auto] {$ \cong $} (m-2-1)
	(m-2-1) edge node[auto] {$ $} (m-2-2)
	(m-1-2) edge node[auto] {$ \cong $} (m-2-2); 
\end{tikzpicture}
\end{equation*}
which commutes.
\end{theorem}
\begin{proof}
For (\ref{tt1}), by definition
$$
M|_{\gen{x_\alpha}} \cong k_{ev}^{\oplus a_{ev}} \oplus k_{od}^{\oplus a_{od}} \oplus
P^{\oplus a_2}
$$
for all $0 \neq \alpha \in \A^s$.
Recall that the fibers of $M \otimes \mc{O}$ are isomorphic $M$ and
that the action of $\theta_M$ on the fibers is multiplication by $x_\alpha$,
so $\dim_{k(\ov{\alpha})} (\Image \theta_M)_{\ov{\alpha}} \otimes k(\ov{\alpha}) = a_2$
for all $\ov{\alpha} \in \P^{s-1}$.
By Theorem \ref{T: vector bundles theorem} (\ref{claim1}), $\ms{F}_2(M) = \Image \theta_M$
is a vector bundle of rank $a_2$.

Consider the short exact sequence
$$
\begin{tikzpicture}[start chain] {
	\node[on chain] {$0$};
	\node[on chain] {$\Image \theta_M $} ;
	\node[on chain] {$\Ker \theta_M$};
	\node[on chain] {$\ms{F}_1(M)$};
	\node[on chain] {$0$}; }
\end{tikzpicture}
$$
which defines the functor $\ms{F}_1(M)$.  Applying the right exact functor
of specialization to the fiber over
$\ov{\alpha}$ yields a diagram
\begin{equation} \label{E: specialization of F}
\begin{tikzpicture}[description/.style={fill=white,inner sep=2pt},baseline=(current  bounding  box.center)]

\matrix (m) [matrix of math nodes, row sep=2.5em,
column sep=3em, text height=1.5ex, text depth=0.25ex]
{
\Image \theta_M \otimes_{\mc{O}} k(\ov{\alpha}) & \Ker \theta_M \otimes_{\mc{O}} k(\ov{\alpha}) & \ms{F}_1(M) \otimes_{\mc{O}} k(\ov{\alpha}) & 0 \\
\Image x_\alpha  &	\Ker x_\alpha & \ms{F}_{1, \alpha}(M) & 0  \\};

\path[->,font=\scriptsize]
	(m-1-1) edge node[auto] {$  $} (m-1-2)
	(m-1-2) edge node[auto] {$  $} (m-1-3)
	(m-1-3) edge node[auto] {$  $} (m-1-4)
	(m-2-1) edge node[auto] {$\iota  $} (m-2-2)
	(m-2-2) edge node[auto] {$  $} (m-2-3)
	(m-2-3) edge node[auto] {$  $} (m-2-4); 
\path[-,double, font=\scriptsize]
	(m-1-1) edge[double,thick,double distance=3pt] node[auto] {$  $} (m-2-1)
	(m-1-2) edge[thick,double distance=3pt] node[auto] {$  $} (m-2-2)
	(m-1-3) edge[thick,double distance=3pt] node[auto] {$  $} (m-2-3);
\end{tikzpicture}
\end{equation}
where $\iota$ is injective for all $\ov{\alpha}$.  Thus, $\ms{F}_{1, \alpha}(M)$
has constant rank $(a_{ev}|a_{od})$ and so by
Theorem \ref{T: vector bundles theorem} (\ref{claim2}), $\ms{F}_1(M)$
is a super vector bundle of rank $(a_{ev}|a_{od})$.

Part (\ref{tt2}) follows from observing that
\begin{equation*}
\begin{tikzpicture}[description/.style={fill=white,inner sep=2pt},baseline=(current  bounding  box.center)]

\matrix (m) [matrix of math nodes, row sep=2.5em,
column sep=3em, text height=1.5ex, text depth=0.25ex]
{
\tilde{M}(j) & \tilde{M}(j+1) \\
\tilde{N}(j) & \tilde{N}(j+1)  \\};

\path[->,font=\scriptsize]
	(m-1-1) edge node[auto] {$ \theta_M $} (m-1-2)
			edge node[left] {$ f \otimes \op{Id}_{\mc{O}}$} (m-2-1)
	(m-2-1) edge node[auto] {$ \theta_N $} (m-2-2)
	(m-1-2) edge node[auto] {$ (-1)^{|f|} f \otimes \op{Id}_{\mc{O}} $} (m-2-2); 
\end{tikzpicture}
\end{equation*}
is a commutative diagram which induces maps $\Image \theta_M \rightarrow \Image \theta_N$
and $\Ker \theta_M \rightarrow \Ker \theta_N$ and hence on the cokernels.
Then applying the specialization diagram in
Equation~\ref{E: specialization of F} completes the proof.
\end{proof}

\subsection{Constant Jordan Type for $\mf{f}_1$}
Another interesting property of $\mf{f}_r$-modules of constant
Jordan type is that when $r = 1$, this Lie superalgebra
is too small to accommodate indecomposable modules of
constant Jordan type where $a_1 > 1$.  One important observation
in showing this fact is that if $M$ is is indecomposable and non-projective,
then $\Rad^2(M) = 0$ and furthermore, the parity of $M$ distinguishes
the head and the socle of $M$, as seen in the following lemma.  Another
key fact is that for such modules, the super commutative action
of $\mf{f}_1$ on $M$ becomes a commutative one since $xy.m = yx.m = 0$.

\begin{lemma} \label{L: ht 2 hd soc lemma}
Let $M$ be an indecomposable non-projective supermodule over
$V(\mf{a}_2) = \gen{1, x, y}$ such that
$\dim(M) > 1$.  Then
\begin{enumerate}
\item $\Rad(M) = \Soc(M)$; \label{rr1}
\item either $\ev{M} = \Hd(M)$ and $\od{M} = \Soc(M)$ or
$\ev{M} = \Soc(M)$ and $\od{M} = \Hd(M)$. \label{rr2}
\end{enumerate}
\end{lemma}
\begin{proof}
For (\ref{rr1}), let $m \in \Rad(M)$.  Then $m = (ax + by).n$ for some $n \in M$
and $(0,0) \neq (a,b) \in k^2$.
Then $x.m = bxy.n$ and $y.m = ayx.n$ both of which must be zero, otherwise
this would be the socle of a projective $V(\mf{a}_2)$-module which would split off
as a direct summand implying $M$ is either projective or decomposable.  Similarly,
if $m' \in \Soc(M)$ then $x.m'$ and $y.m'$ are both zero and since $M$ is indecomposable,
$m'$ is the image of some $n' \in M$ under the action of $V(\mf{a}_2)$ which proves
(\ref{rr1}).

We show (\ref{rr2}) by induction on $\dim(M)$.  The base case is when $\dim(M) = 2$,
which is trivial since $M$ is indecomposable.  Let $\dim(M) > 2$ and let $m \in \Hd(M)$
and let $N$ denote the maximal submodule of $M$ complimentary to $m$.  Then
$\dim(N) < \dim(M)$ and so by the inductive hypothesis, $N$ satisfies (\ref{rr2}).
Since $M$ is indecomposable, $(ax + by).m \neq 0$ for some $(0,0) \neq (a,b) \in k^2$.
Then $(ax + by).m \in \Soc(M) \cap N$ and since the action of $ax + by$ is odd,
then $m$ has the opposite parity as $\Soc(N)$ and thus has the same parity as
$\Hd(N)$ which implies that (\ref{rr2}) holds for $M$ as well.
\end{proof}

\begin{proposition} \label{P: super JT is concentrated}
Let $M$ be an indecomposable
non-projective supermodule over
$\mf{f} = \sL$ of strong constant Jordan type
$(a_{ev}|a_{od})[1] + a_2[2]$.  Then exactly one of $a_{ev}$ and
$a_{od}$ is 0, i.e. there cannot be both even and odd trivial
summands in the Jordan decomposition of $M$.
\end{proposition}
\begin{proof}
First note that since $M$ is not projective, then
$a_{ev}+ a_{od} = a_1 \neq 0$
by Proposition \ref{P: CJT has max atyp}.  Furthermore,
if $a_2 = 0$ then $a_1 = 1$ since $M$ is indecomposable and
so these cases follow immediately.  In order to utilize Lemma
\ref{L: ht 2 hd soc lemma},
we consider the strong Jordan decomposition of $M$,
$$
M|_{\gen{ax+by}} \cong
k_{ev}^{\oplus a_{ev}} \oplus k_{od}^{\oplus a_{od}} \oplus P^{\oplus a_2},
$$
for all $0 \neq ax + by \in \od{\mf{f}}$, and we have reduced to considering
the situation when both $a_{ev} + a_{od} \geq 1$ and $a_2 \geq 1$.
In the remainder of the proof, by dualizing and then applying
the parity change functor (if either are necessary),
we assume without
loss of generality that
\begin{enumerate}
\item $\dim \Hd(M) \leq \dim \Soc(M)$; \label{ss1}
\inlineitem $\ev{M} = \Hd(M)$ and $\od{M} = \Soc(M)$. \label{ss2}
\end{enumerate}
This further implies that
\begin{enumerate}
\setcounter{enumi}{2}
\item $\dim \Hd(M) = a_{ev} + a_2$; \label{ss3}
\inlineitem $\dim \Soc(M) = a_{od} + a_2$; \label{ss4}
\inlineitem $a_{ev} \leq a_{od}$. \label{ss5}
\end{enumerate}

As noted above, the action of $\mf{f}$ on $M$ becomes commutative since
$\Rad^2(M) = 0$ so we can apply the theory of generic kernels and images
introduced in \cite{CFS} and summarized in \cite[Chapter 4]{Benson-Book}.
Additionally, in this situation since $M$ is of strong constant Jordan type
it is consistent with the notion of constant Jordan type
in \cite{Benson-Book}.
Let $\mf{K}(M)$ denote the generic kernel of $M$.

Using \cite[Lemma 4.10.12]{Benson-Book}, we can determine exactly
$\mf{K}(M)$.  Define
\begin{gather*}
n := \text{the number of Jordan blocks of each $0 \neq ax + by \in \mf{f}$ acting on $M$} \\
d:= \dim \mf{K}(M)/ \Rad (\mf{K}(M))
\end{gather*}  
and the Lemma yields that $n = d$.
Note that our decomposition gives us that $n = a_{ev} + a_{od} + a_2$.

Applying \cite[Proposition 4.7.8]{Benson-Book} to $M$ yields
$\Soc(M) \subseteq \mf{K}(M)$.
By \cite[Theorem 4.7.4]{Benson-Book}, $\mf{K}(M)$ has the constant image
property and \cite[Proposition 5.1]{CFS} gives a classification of
such modules.  Again, the result from \cite{CFS} applies here as it
is based on \cite[Proposition 5]{HR} which applies to the general
setting we consider here.
By the classification,
$$
\mf{K}(M) \cong W_{n_1, 2} \oplus \dots W_{n_t, 2} \oplus k_{od}^s
$$
for some integers $s, t, n_i$, where the $W_{n_i, 2}$ are defined
in \cite[Section 4.11]{Benson-Book}.  Note that $s$ cannot be 0 or else
condition (\ref{ss1}) above is violated.  Additionally, the image of $\mf{K}(M)$
under the action of $\mf{f}$ is exactly $\oplus_i \Ker(W_{n_i, 2})$,
so if $t \geq 1$, then since $\Hd(W_{n_i,2})$ is not in the image of any
element of $M$, then this produces a direct sum decomposition of $M$,
a contradiction.  Thus, $\mf{K}(M) \cong k_{od}^s$ and since
$\Soc(M) \subseteq \mf{K}(M)$, we conclude that
in fact $\Soc(M) = \mf{K}(M)$.
Then $d_1 = a_{od} + a_2$
by (\ref{ss4}) and so $a_{od} + a_2 = a_{ev} + a_{od} + a_2$ and
therefore $a_{ev} = 0$.  Recalling that
$a_{ev}+a_{od} \geq 1$ and that the reductions made may have
changed the parity, we conclude that
exactly one of $a_{ev}$ and $a_{od}$ is 0.
\end{proof}

Combining the results in \cite[Chapter 4]{Benson-Book}
with Proposition \ref{P: super JT is concentrated}, the following
theorem completely classifying modules of constant Jordan type
over $\mf{f}$ follows quickly.

\begin{theorem} \label{T: cjt class for f1}
Let $M$ be an indecomposable
non-projective supermodule over
$\mf{f} = \sL$ of strong constant Jordan type
$a_1[1] + a_2[2]$.  Then $a_1 = 1$ and hence $M$ is endotrivial.
\end{theorem}
\begin{proof}
Begin by making the same assumptions (\ref{ss1}) - (\ref{ss5}) as
in the previous proof and thus $a_{ev} = 0$.
Furthermore, assume that $M$ is not $k_{od}$ as this case is trivial.

Let $\mf{I}(M)$ denote the generic image of $M$ as defined in
\cite[Section 4.10]{Benson-Book}.  According to the
results from the same section, $\mf{I}(M) \subseteq \mf{K}(M) = \Soc(M)$,
$M/ \mf{I}(M)$ has (ungraded) constant Jordan type,
and for all $0 \neq (a,b) \in k^2$, $\mf{I}(M) \subseteq \Image(ax +by)$.

Assume that $\mf{I}(M) \neq 0$.  Then the fact that
$\mf{I}(M) \subseteq \Image(ax +by)$ for all $0 \neq (a,b) \in k^2$
implies that $\mf{I}(M) \subset \Soc(P^{\oplus a_2})$
in the Jordan decomposition of $M$ relative to $ax + by$.  Thus,
$M/ \mf{I}(M)$ is of constant Jordan type as a supermodule, and
has type
$$
(\dim(\mf{I}(M) | a_{od})[1] + (a_2 - \dim(\mf{I}(M))[2]
$$
which is a contradiction, so $\mf{I}(M) = 0$.

By dualizing, \cite[Lemma 4.10.3]{Benson-Book} implies that 
$$
M^* \cong \mf{K}(M^*) \cong W_{n_1, 2} \oplus \dots W_{n_t, 2} \oplus k_{od}^s
$$
and since $M$ and therefore $M^*$ are indecomposable, $M \cong W^*_{n,2}$ for some
$n \geq 0$.

Recalling the reductions we made, $M$ may be isomorphic to $W_{n,2}$
or $W^*_{m,2}$ which both have constant Jordan type $1[1] + (n-1)[2]$,
and the trivial summand may be concentrated in either degree.
\end{proof}

%
%

Note that there are infinitely many nonisomorphic
indecomposable endotrivial $\mf{f}_1$-modules as proved in
\cite{Talian-2013} (isomorphic to the $W_{n,2}$ and $W^*_{n,2}$)
and that for $\mf{f}_r$ when
$r > 1$, there are indeed indecomposable
modules of constant Jordan type $(a_{ev}|a_{od})[1] + a_2[2]$
where $a_{ev}, a_{od} \geq 1$ as constructed
in Section \ref{SS: syzygy mod soc}.

\section*{Acknowledgments}
The author is grateful to Institut Mittag-Leffler for their hospitality
and support during the Representation Theory Program held
at the institute in the spring of 2015 where the majority
of this work was completed.  The author would also like to thank
Jonathan Kujawa for many fruitful discussions in regard to this
paper which helped direct and focus the ideas presented here.

\end{document}